\documentclass[11pt,reqno]{amsart}
\usepackage{amsmath, amsfonts, amsthm, amssymb,amscd, graphicx, amscd}
\usepackage{float,epsf}
\usepackage[english]{babel}
\usepackage{enumerate}
\usepackage{tikz}
\usepackage{mathrsfs}
\usepackage[numbers,sort&compress]{natbib}

\usepackage{srcltx}
\usepackage{geometry}
\usepackage{verbatim}
\usepackage{mathrsfs}
\usepackage{hyperref}
\usepackage{enumitem} 
\usepackage{bbm} 
\usepackage{stmaryrd}

\usepackage{relsize}
\usepackage{exscale}

\usepackage{mathtools}

\newtheorem{thm}{Theorem}[section]
\newtheorem{defn}[thm]{Definition}

\newtheorem{lem}[thm]{Lemma}
\newtheorem{cor}[thm]{Corollary}

\newtheorem{rem}[thm]{Remark}

\numberwithin{equation}{section}

\def\dd{{\rm d}}
\hypersetup{bookmarksdepth=2}

\begin{document}
\title{NONLOCAL-TO-LOCAL LIMIT FOR LINEAR TRANSPORT EQUATIONS WITH MEASURE INITIAL DATA}
\author{Immanuel Ben Porat}
\begin{abstract}
We initiate the study of the nonlocal to local limit for linear transport equations. A nonlocal version of the linear transport equation is introduced alongside an appropriate well-posedness theory for distributional solutions with measure initial data. We proceed by deriving the associated linear transport equation as a nonlocal-to-local limit. Some of the results apply in arbitrary dimension and for measure initial data.        
\end{abstract}
\maketitle
\section{Introduction}
The nonlocal-to-local limit seeks to derive a local transport equation from an associated  nonlocal equation. The local transport equation can be either linear or nonlinear, i.e. a Burgers type equation. The nonlocal associated PDE  is a delocalized version of the transport or Burgers equation, respectively. In this work, we study the nonlocal-to-local limit for linear transport equations. Let us start by introducing the nonlocal version of the linear transport equation in non-conservative form.  Given an integrable function  $\eta\in L^{1}(\mathbb{R}^{d})$ with $\int_{\mathbb{R}^{d}}\eta(x)\ \dd x=1$ and  $\varepsilon>0$ set $$\eta_{\varepsilon}(x)\coloneqq \frac{1}{\varepsilon^{d}}\eta(\frac{x}{\varepsilon}).$$ 
With this notation consider the following Cauchy problem 
\begin{align}
\begin{cases}
\begin{array}{lc}
\partial_{t}u_{\varepsilon}+B\cdot \nabla_{x}(\eta_{\varepsilon}\ast u_{\varepsilon})=0, \ (t,x)\in (0,\infty)\times \mathbb{R}^{d}\\
u_{\varepsilon}(0,x)=u^{0}, \ x\in \mathbb{R}^{d} .
\end{array}\end{cases}
\label{nonlocal Cauhy problem simplification}
\end{align}
Here $B:\mathbb{R}^{d}\rightarrow \mathbb{R}^{d}$ is a given vector field. Since $\eta_{\varepsilon}$ is an approximation of the identity, it converges weakly to a Dirac delta. Therefore, at least formally, we expect to recover the following \textbf{local linear transport equation} in the limit as $\varepsilon\rightarrow 0$: 
\begin{align}
\begin{cases}
\partial_{t}u+B\cdot \nabla_{x}u=0, \ (t,x)\in (0,\infty)\times \mathbb{R}^{d}\\
u(0,x)=u^{0}, \ x\in \mathbb{R}^{d}. 
\end{cases}  
\label{local Cauchy}
\end{align} 
The rigorous justification of this limit has been mostly studied in the context of the passage from nonlocal conservation laws to Burgers-type equations. The objective of the present work is to derive linear transport equations as nonlocal limits. Some of our methods are inspired by the existing literature on nonlocal-to-local limits for Burgers-type equations, and we therefore briefly review some of the main works on the subject.
\subsection{The nonlocal to local limit for Burgers type equations}  
Traditionally the nonlocal-to-local limit has been studied for nonlocal conservation laws, which frequently arise in traffic flows. More specifically, given a function $V:\mathbb{R}\rightarrow \mathbb{R}$ (often referred to as a velocity function) consider the scalar Cauchy problem 
\begin{align}
\begin{cases}
\begin{array}{lc}
\partial_{t}u_{\varepsilon}+\partial_{x}(V(\eta_{\varepsilon}\ast u_{\varepsilon})u_{\varepsilon})=0, \ (t,x)\in (0,\infty)\times \mathbb{R}\\
u_{\varepsilon}(0,x)=u^{0},\ x\in \mathbb{R}.
\end{array}
\end{cases}
\label{nonlocal conservation}
\end{align}
As $\varepsilon\rightarrow 0$ one expects to derive the following local nonlinear conservation law 
\begin{align}
\begin{cases}
\begin{array}{lc}
\partial_{t}u+\partial_{x}(V(u)u)=0, \ (t,x)\in (0,\infty)\times \mathbb{R}\\
u(0,x)=u^{0}, \ x\in \mathbb{R}. 
\end{array}\end{cases}   
\label{local Burgers}
\end{align}
Global and local well-posedness theory of the initial value problem \eqref{nonlocal conservation} and its various variants is well established   - see \cite{keimer2023discontinuous,blandin2016well, crippa2013existence, coclite2022existence,colombo2024multidimensional,keimer2017existence,keimer2018existence, coclite2025singular, keimer2019nonlocal}. Moreover, under appropriate assumption on $V$ and $\eta$ it is possible to show that \eqref{nonlocal conservation} verifies a maximum principle and hence $\left\Vert u_{\varepsilon}\right\Vert_{L^{\infty}([0,T]\times \mathbb{R})}\leq \left\Vert u^{0}\right\Vert_{L^{\infty}(\mathbb{R})}$. As a result, $u_{\varepsilon}$ admits a weak-$\ast$ limit  $u\in L^{\infty}(\mathbb{R}_{+}\times \mathbb{R})$ (up to a subsequence) by the Banach-Alaoglu theorem. However, this does not automatically imply that   $V(\eta_{\varepsilon}\ast u_{\varepsilon})u_{\varepsilon}\underset{\varepsilon\rightarrow 0}{\rightharpoonup}V(u)u$, because weak convergence is in general ill-behaved under products. This is one reason which does not allow to immediately identify $u$ as a weak distributional solution of \eqref{local Burgers}. Even after this obstacle has been surpassed, an additional difficulty is in showing that this distributional solution is in fact the unique entropy admissible solution for \eqref{local Burgers}. Identifying  the precise assumptions which need to be imposed on $\eta,V$ and $u^{0}$ for this convergence to be valid is a main challenge in the nonlocal-to-local limit problem. 
Among several important works in recent years, which include both positive results as well as counterexamples to the nonlocal-to-local limit we mention \cite{colombo2019singular, colombo2021local,coclite2024oleinik, friedrich2024conservation, coclite2025singular,de2026volterra, coclite2022general}. See also \cite{colombo2023overview} for a general overview of the state of the art of the problem. 
\subsection{The nonlocal to local limit for linear transport equations.} The main aim of this work is to study the nonlocal-to-local limit leading from \eqref{nonlocal Cauhy problem simplification} to \eqref{local Cauchy}. Naturally, this also leads us to develop an appropriate well-posedness theory for signed distributional solutions of \eqref{nonlocal Cauhy problem simplification} - this will be the content of section \ref{Well posedness sec}. As already mentioned the literature on the non-local to local limit for linear transport equations appears to be limited, although we mention \cite{coclite2026nonlocal} which provides a derivation of a one-dimensional linear transport equation coupled to a Burgers equation. See also the recent \cite{gong2026existence} which focuses on the existence and uniqueness of generalizations of \eqref{nonlocal conservation}, which in particular fall in the scope of equations of the type \eqref{nonlocal Cauhy problem simplification}.   The novelty of our work is that we manage to prove the nonlocal-to-local limit for transport equations under assumptions which are significantly more general than what is  available in the aforementioned works. In particular our results generalize the nonlocal -to-local limit from several  viewpoints:  
\begin{enumerate}
    \item Dimensional: Our first two main results offer a derivation valid for arbitrary dimensions. This is at variance with the previously mentioned literature, which utilizes genuinely one-dimensional arguments. The only exception is \cite{coclite2025singular}, which studies the nonlocal-to-local limit for viscous conservation laws, leveraging the regularizing effect of the viscosity term. When no viscosity is included, the nonlocal-to-local limit in dimension higher than $1$ is largely open.   
    \item Measure initial data: The derivations offered by Theorems \ref{Constant transport limit}-\ref{nonlocal to local general B} are valid for distributional solutions of \eqref{local Cauchy} with measure initial data. 
    \item General kernels: At least in the case  where \eqref{local Cauchy} takes the form of a constant transport equation (Theorem \ref{Constant transport limit})), we cover a fairly general class of approximations of identities. This is in contrast to anisotropy assumptions imposed on the kernel in the aforementioned works, or the symmetry assumptions imposed in \cite{zumbrun1999nonlocal} and \cite{ghoshal2024non}. However, in Theorem \ref{Free transport limit} we need to impose a symmetry assumption and in Theorem \ref{nonlocal to local general B} we need the kernel to be anisotropic.    
    \item Signed initial data: The assumption that the initial data has a distinguished sign (which is preserved in time) is crucial in many works on the nonlocal-to-local limit for Burgers type equations. This assumption is redundant in Theorems \ref{Constant transport limit}-\ref{Free transport limit}, but not in Theorem \ref{nonlocal to local general B}. 
\end{enumerate}
Our first two main results concern the nonlocal-to-local limit for the two most common models of linear transport, namely the \textbf{constant transport  equation} and \textbf{free transport/Liouville equation}.  
\begin{thm}\textup{(\textbf{Nonlocal to local limit for constant transport})}
\\ Suppose that: 
\begin{itemize}
    \item $u^{0}\in H^{-s}(\mathbb{R}^{d})\cap \mathcal{M}(\mathbb{R}^{d})$ for some $s>\frac{d}{2}$. 
    \item $\eta\in \mathrm{BV}(\mathbb{R}^{d})$ is such that $\int_{\mathbb{R}^{d}}\eta(x)\ \dd x=1$. Moreover at least one of the following holds:
    \begin{itemize}
        \item     For all $\xi\in \mathbb{R}^{d}$ it holds that $\Re(\widehat{\eta}(\xi))\neq 0$.
        \item $\Im(\eta)\equiv 0$. 
    \end{itemize}
    \item $B=(B_{1},\dots,B_{d})\in \mathbb{R}^{d}$ is a constant vector. 
\end{itemize}
Let $u_{\varepsilon}$ be the unique distributional  solution of \eqref{nonlocal Cauhy problem simplification} with initial data $u^{0}$ (as guaranteed via Theorem \ref{measure valued solutions existence and uniqueness}) and let $u$ be the unique distributional solution of \eqref{local Cauchy} with initial data $u^{0}$. Then, for any $T>0$ it holds that 
\begin{align*}
\underset{t\in [0,T]}{\sup}\left\Vert (u_{\varepsilon}-u)(t,\cdot)\right\Vert_{H^{-s}}\underset{\varepsilon \rightarrow 0}{\rightarrow} 0.    
\end{align*}
\label{Constant transport limit}
\end{thm}
\begin{thm}\textup{(\textbf{Nonlocal to local limit for free transport})}\\
Suppose that: 
\begin{itemize}
    \item $u^{0}\in H^{-s}(\mathbb{R}^{d}_{x}\times \mathbb{R}^{d}_{v})\cap \mathcal{M}_{c}(\mathbb{R}^{d}_{x}\times \mathbb{R}^{d}_{v})$ for some $s>1+\frac{d}{2}$. Moreover, suppose that $\widehat{u}^{0}(\xi,\zeta)\neq 0$ for all $(\xi,\zeta)$. 
    \item $\eta\in \mathcal{S}(\mathbb{R}^{d}_{x}\times \mathbb{R}^{d}_{v})$ is even \textup{(}$\eta(x,v)=\eta(-x,-v
    )$\textup{)} and $\int_{\mathbb{R}^{d}\times \mathbb{R}^{d}}\eta(x,v)\ \dd x\dd v=1$. 
    \item $B(x,v)=(v,\mathbf{0})=(v_{1},\dots,v_{d},0,\dots,0)$. 
\end{itemize}
Let $u_{\varepsilon}$ be the unique distributional solution of \eqref{nonlocal Cauhy problem simplification} with initial data $u^{0}$ (as guaranteed via Theorem \ref{Measure valued solution to Liouville eq}) and let $u$ be the unique distributional solution of \eqref{local Cauchy} with initial data $u^{0}$. Then, for any $T>0$ it holds that 
\begin{align*}
\underset{t\in [0,T]}{\sup}\left\Vert (u_{\varepsilon}-u)(t,\cdot)\right\Vert_{H^{-s}}\underset{\varepsilon \rightarrow 0}{\rightarrow} 0.   
\end{align*}
\label{Free transport limit}
\end{thm}
A few remarks are in order after the statement of Theorems \ref{Constant transport limit}-\ref{Free transport limit}. The reason we focus on these specific equations (where $B$ is either constant or linear) is related to the fact that in both cases the Fourier transform of the solution is governed by an explicit linear transport equation, a fact which will be of crucial importance both in the passage to the limit as well as in certain parts of the well-posedness theory. It is a curious question whether it is possible to extend the results to more general $B$. The requirement that $\eta$ is even in Theorem \ref{Free transport limit} is to ensure that the Fourier transform of $\eta$ is real-valued, which will be used to study the linear transport equation governing the Fourier transform of the solution. It should also be noted that we expect that if more regularity is available for the initial data $u^{0}$, then it should be possible to improve the mode of convergence from weak convergence to strong convergence.
Finally, we are able to treat arbitrary non-positive $B$ in the 1D case when $\eta$ is anisotropic and the initial data is a non-negative measure. Precisely, we have the following result. 
\begin{thm} \textup{(\textbf{Nonlocal to local limit for arbitrary $B\leq 0$})}\\
Suppose that: 
\begin{itemize}
    \item $u^{0}\in H^{-s}(\mathbb{R})\cap \mathcal{M}_{+}(\mathbb{R})$ for some $s>\frac{1}{2}$. 
    \item $B\in W^{s,\infty}(\mathbb{R})\cap C^{\infty}(\mathbb{R})$ and $B(x)\leq 0$ for all $x\in \mathbb{R}$. 
    \item $\eta\in L^{1}(\mathbb{R})\cap L^{\infty}(\mathbb{R})$, $\mathrm{supp}(\eta)\subset \mathbb{R}_{-}$, $\eta$ is non-decreasing on $\mathbb{R}_{-}$ and $\eta\geq 0$ satisfies $\int_{\mathbb{R}}\eta(x)\ \dd x=1$ (in particular, $\eta\in \mathrm{BV}(\mathbb{R})$).  
\end{itemize}
Let $u_{\varepsilon}$ be the unique distributional  solution to \eqref{nonlocal Cauhy problem simplification} with initial data $u^{0}$ (as guaranteed via Theorem \ref{measure valued solutions existence and uniqueness}) and let $u$ be the unique distributional solution of \eqref{local Cauchy} with initial data $u^{0}$. 
Then, for any $T>0$ it holds that 
\begin{align*}
\left\Vert (u_{\varepsilon}-u)(t,\cdot)\right\Vert_{L^{2}([0,T];H^{-s}(\mathbb{R}))}\underset{\varepsilon \rightarrow 0}{\rightarrow}0.  
\end{align*}
\label{nonlocal to local general B}
\end{thm}
The paper is organized as follows. In section \ref{Well posedness sec} we provide the appropriate well-posedness theory for \eqref{nonlocal Cauhy problem simplification}. The methods involved for the case where $B$ is bounded differ substantially from the methods required to handle linear $B$ (as in Theorem \ref{Free transport limit}). The latter is based on Fourier-transforming the equation, a tool which is redundant when $B$ is bounded. In section \ref{Proof of main theorems sec} we further exploit the Fourier approach used in section \ref{Well posedness sec} in order to prove the nonlocal-to-local limit for the choices of $B$ indicated in Theorems \ref{Constant transport limit} and \ref{Free transport limit}. Finally, in section \ref{general B sec} we prove Theorem \ref{nonlocal to local general B}, which proves the nonlocal-to-local limit for the case of general non-positive $B$. This is achieved  in the settings of anisotropic kernels. \section{Well-posedness theory}\label{Well posedness sec}
In this section we prove the well-posedness of \eqref{nonlocal Cauhy problem simplification} both for Lipschitz bounded $B$, as well as for $B$ as in Theorem \ref{Free transport limit} (which has linear growth). We always take $\varepsilon=1$ in this section as it will not be relevant for the well-posedness theory.   
We denote by $\mathcal{M}(\mathbb{R}^{d})$ (resp. $\mathcal{M}_{+}(\mathbb{R}^{d})$) the space of signed finite Radon measures (resp. positive finite Radon measures). We denote by $\mathcal{M}_{c}(\mathbb{R}^{d})$ the space of finite Radon measures with compact support. Given $u\in \mathcal{M}({\mathbb{R}}^{d})$ denote by $\left\vert u\right\vert$ the total variation of $u$. In case we wish to emphasize a function $f$ is complex valued we will write $f\in L^{2}(\mathbb{R}^{d};\mathbb{C}),f\in \mathcal{S}(\mathbb{R}^{d};\mathbb{C}),\dots$ etc.  
\begin{defn}\textup{(\textbf{Distributional solutions of \eqref{local Cauchy}}}\\
Suppose that: 
\begin{itemize}
    \item $u^{0}\in H^{-s}(\mathbb{R}^{d})\cap \mathcal{M}(\mathbb{R}^{d})$. 
    \item $B\in \mathrm{Lip}(\mathbb{R}^{d})\cap C^{\infty}(\mathbb{R}^{d})$.
\end{itemize}
We call $u\in L^{\infty}([0,T];H^{-s}(\mathbb{R}^{d}))$ a distributional solution of \eqref{local Cauchy} with initial data $u^{0}$ if for any $\varphi\in C^{\infty}_{0}((-a,T)\times \mathbb{R}^{d})$ ($a>0$ arbitrarily small) it holds that 
\begin{align*}
\int_{0}^{T}\langle \partial_{\tau}\varphi(\tau,\cdot),u(\tau,\cdot)\rangle+\langle\mathrm{div}_{x}(B\varphi)(\tau,\cdot),u(\tau,\cdot)\rangle \ \dd \tau +\langle \varphi(0,\cdot),u^{0}\rangle =0.    
\end{align*}
\end{defn}
The existence and uniqueness of distributional solutions of \eqref{local Cauchy} with measure initial data is classical. Next, we define distributional solutions for the nonlocal equation \eqref{nonlocal Cauhy problem simplification}. Given $\eta\in L^{1}(\mathbb{R}^{d})$ denote by $\eta^{\mathrm{odd}}$ and $\eta^{\mathrm{even}}$ the odd and even parts of $\eta$, respectively.  Explicitly, $\eta^{\mathrm{odd}}$ and $\eta^{\mathrm{even}}$ are defined by  
\begin{align*}
\eta^{\mathrm{odd}}(x)\coloneqq\frac{\eta(x)-\eta(-x)}{2}  \ \mbox{and}\ \eta^{\mathrm{even}}(x)\coloneqq\frac{\eta(x)+\eta(-x)}{2}.     
\end{align*}
We also define $\overline{\eta}(x)\coloneqq \eta^{\mathrm{even}}(x)-\eta^{\mathrm{odd}}(x)=\eta(-x)$. 
\begin{defn}\label{Defintion measure valued for nonlocal}\textup{(\textbf{Distributional solutions of \eqref{nonlocal Cauhy problem simplification})}}\\
Suppose that: 
\begin{itemize}
    \item $u^{0}\in H^{-s}(\mathbb{R}^{d})\cap \mathcal{M}(\mathbb{R}^{d})$.
    \item $B\in \mathrm{Lip}(\mathbb{R}^{d})\cap C^{\infty}(\mathbb{R}^{d})$. 
    \item $\eta\in \mathrm{BV}(\mathbb{R}^{d})$.
\end{itemize}
We call $u\in L^{\infty}([0,T];H^{-s}(\mathbb{R}^{d}))$ a distributional solution of \eqref{nonlocal Cauhy problem simplification} with initial data $u^{0}$ if for any $\varphi\in C^{\infty}_{0}((-a,T)\times \mathbb{R}^{d})$  it holds that 
\begin{align}
&\int_{0}^{T}\langle\partial_{\tau}\varphi(\tau,\cdot),u(\tau,\cdot) \rangle+ \langle\mathrm{div}_{x}(\overline{\eta}\ast (B\varphi))(\tau,\cdot),u(\tau,\cdot)\rangle \ \dd \tau
+\langle \varphi(0,\cdot),u^{0}\rangle =0. \notag\\
. \label{weak formulation for nonlocal}      
\end{align}    
\end{defn}
\begin{rem}
Note the following: 
\begin{itemize}
\item The above definition has been devised in order to circumvent the fact that $\eta\ast u$ is not well defined when $u$ is a tempered distribution and $\eta\in \mathrm{BV}(\mathbb{R}^{d})$. This is the reason why we let the convolution of $\overline{\eta}$ bare on $B\varphi$, motivated by the identity $\int (\eta\ast f) g=\int f(\overline{\eta}\ast g)$. We insist on imposing minimal regularity on $\eta$ since eventually we will invoke this concept of solution in the case where $\eta$ is anisotropic (and as such, exhibits $\mathrm{BV}$ jump discontinuities). 
\item All the integrals in  \eqref{weak formulation for nonlocal} are well defined. Indeed, if $\rho\in \mathrm{BV}(\mathbb{R}^{d})$ then for some compact $K\subset \mathbb{R}^{d}$ it holds that   
\begin{align*}
\left\Vert \mathrm{div}_{x}(\rho\ast (B\varphi))(\tau,\cdot)\right\Vert_{H^{s}}\lesssim \left\Vert \rho\right\Vert_{\mathrm{BV}}\left\Vert B\right\Vert_{W^{s,\infty}(K)} \left\Vert \varphi(\tau,\cdot)\right\Vert_{H^{s}}.   \end{align*}
\item If in addition $\eta\in\mathcal{S}(\mathbb{R}^{d})$ in Definition \ref{Defintion measure valued for nonlocal} then the left-hand side of \eqref{weak formulation for nonlocal} can be written alternatively as 
\begin{align*}
\int_{0}^{T}\langle \partial_{\tau}\varphi(\tau,\cdot), u(\tau,\cdot)\rangle\ \dd \tau -\int_{0}^{T}\int_{\mathbb{R}^{d}}B\cdot \nabla_{x}\eta\ast  u(\tau,x)\varphi(\tau,x)\ \dd x\dd \tau+\langle \varphi(0,\cdot), u^{0}\rangle=0.
\end{align*}
\end{itemize}
\end{rem}
We start with the case where $B$ is bounded - the case where $B$  is linear in $v$ (and thus fails to be bounded) will be handled afterwards through a different method. We prove well-posedness for classical solutions and then prove well-posedness for distributional solutions. 
\begin{thm} \textup{(\textbf{Classical solutions for \eqref{nonlocal Cauhy problem simplification}}).}
\label{wellposedness classical sol}
\\
Suppose that: 
\begin{itemize}
\item $u^{0}\in W^{1,\infty}(\mathbb{R}^{d})$.
\item $B\in W^{1,\infty}( \mathbb{R}^{d})\cap C^{\infty}(\mathbb{R}^{d})$. 
\item $\eta\in  \mathrm{BV}(\mathbb{R}^{d})$.
\end{itemize}
Then, there exist a unique classical solution $u\in W^{1,\infty}([0,T]\times \mathbb{R}^{d})$ to \eqref{nonlocal Cauhy problem simplification} with initial data $u^{0}$. 
\end{thm}
\begin{proof}
We aim to show the existence and uniqueness of a solution $u\in L^{\infty}([0,T]\times \mathbb{R}^{d})$ for the integral equation 
\begin{align}
 u(t,x)=u^{0}(x)+\int_{0}^{t}B(x)\cdot \nabla_{x}\eta\ast u(\tau,x)\ \dd \tau.   \label{integral equation} 
\end{align}
Let $T_{\ast}=\frac{1}{2\max\{1,\left\Vert \eta\right\Vert_{\mathrm{TV}} \left\Vert B\right\Vert_{\infty} \}}$ and define
\begin{align*}
\mathfrak{X}\coloneqq \left\{u\in L^{\infty}([0,T_{\ast}]\times \mathbb{R}^{d}):u(0,\cdot)=u^{0}\right\}.   \end{align*}
Given $F\in \mathfrak{X}$ define   
\begin{align*}
u[F](t,x)\coloneqq u^{0}(x)+\int_{0}^{t}B(x)\cdot \nabla_{x}\eta\ast F(\tau,x)\ \dd \tau.   \end{align*}
Since $B\in L^{\infty}(\mathbb{R}^{d})$ it is evident that $u[F]\in \mathfrak{X}$ provided $F\in \mathfrak{X}$. We assert that $F\mapsto u[F]$ is a contraction from $\mathfrak{X}$ to itself. Given $F_{1},F_{2}\in \mathfrak{X}$ we abbreviate $u_{1}=u[F_{1}]$ and $u_{2}=u[F_{2}]$.  
Then, it holds that 
\begin{align*}
(u_{1}-u_{2})(t,x)=\int_{0}^{t}B(x)\cdot \nabla_{x}\eta\ast 
(F_{1}-F_{2})(\tau,x) \ \dd \tau.
\end{align*} 
Therefore it follows that 
\begin{align*}
\left\Vert u_{1}-u_{2}\right\Vert_{L^{\infty}([0,T_{\ast}]\times \mathbb{R}^{d})}\leq T_{\ast}\left\Vert \eta\right\Vert_{\mathrm{TV}}\left\Vert B\right\Vert_{\infty}\left\Vert F_{1}-F_{2}\right\Vert_{L^{\infty}([0,T_{\ast}]\times \mathbb{R}^{d})}= \frac{1}{2}\left\Vert F_{1}-F_{2}\right\Vert_{L^{\infty}([0,T_{\ast}]\times \mathbb{R}^{d})} .  \end{align*}
We have thus proved that $F\mapsto u[F]$ is a contraction and therefore by the Banach contraction principle there is a unique solution $u\in L^{\infty}([0,T_{\ast}]\times \mathbb{R})$ to the equation 
\begin{align}
u(t,x)=u^{0}(x)+\int_{0}^{t}B(x)\cdot \nabla_{x}\eta\ast u(\tau,x)\ \dd \tau. \label{integral version}  \end{align}
By a standard iteration argument we conclude the existence of a unique solution $u\in L^{\infty}([0,T]\times \mathbb{R}^{d})$ to \eqref{integral equation} on the entire time interval $[0,T]$. \\
To show that $u\in W^{1,\infty}([0,T]\times \mathbb{R}^{d})$ we first assume in addition that $\eta\in C^{\infty}_{0}(\mathbb{R}^{d})$. In this case it follows  immediately from \eqref{integral version} that $u\in W^{1,\infty}([0,T]\times \mathbb{R}^
{d})$. Therefore taking the gradient in $x$ in \eqref{integral version} we find that 
\begin{align*}
\left\Vert u(t,\cdot)\right\Vert_{W^{1,\infty}}&=\left\Vert u(t,\cdot)\right\Vert_{\infty} +\left\Vert \nabla_{x}u(t,\cdot)\right\Vert_{\infty}\\
&\leq \left\Vert u^
{0}\right\Vert_{W^{1,\infty}}+\left\Vert B\right\Vert_{\infty}\left\Vert \eta\right\Vert_{\mathrm{TV}} \int_{0}^{t}\left\Vert u(\tau,\cdot)\right\Vert_{\infty}\ \dd \tau   \\&+\left\Vert D_{x}B\right\Vert_{\infty}\left\Vert \eta\right\Vert_{\mathrm{TV}}\int_{0}^{t}\left\Vert u(\tau,\cdot)\right\Vert_{\infty}\ \dd \tau
+\left\Vert B\right\Vert_{\infty}\left\Vert \eta\right\Vert_{\mathrm{TV}}\int_{0}^{t}\left\Vert \nabla_{x}u(\tau,\cdot)\right\Vert_{\infty}\ \dd \tau\\
&\leq \left\Vert u^{0}\right\Vert_{W^{1,\infty}}+2\left\Vert \eta\right\Vert_{\mathrm{TV}}\left\Vert B\right\Vert_{W^{1,\infty}}\int_{0}^{t}\left\Vert u(\tau,\cdot)\right\Vert_{W^{1,\infty}}\ \dd \tau.   
\end{align*}
By Gr\"onwall's lemma it follows that 
\begin{align}
\left\Vert u(t,\cdot)\right\Vert_{W^{1,\infty}}\leq e^{2t\left\Vert \eta\right\Vert_{\mathrm{TV}}\left\Vert B\right\Vert_{W^{1,\infty}} }\left\Vert u^{0}\right\Vert_{W^{1,\infty}}.\label{W1infty in x}       
\end{align}
Consequently it also follows that 
\begin{align}
\left\Vert \partial_{t}u\right\Vert_{\infty}\leq \left\Vert B\right\Vert_{\infty}\left\Vert \eta\right\Vert_{1}e^{2T\left\Vert \eta\right\Vert_{\mathrm{TV}}\left\Vert B\right\Vert_{W^{1,\infty}} }\left\Vert u^{0}\right\Vert_{W^{1,\infty}}. \label{lip in t est}       
\end{align}
To conclude, gathering \eqref{W1infty in x}-\eqref{lip in t est} we get 
\begin{align}
\left\Vert u\right\Vert_{W^{1,\infty}_{t,x}}\leq e^{2T\left\Vert \eta\right\Vert_{\mathrm{TV}}\left\Vert B\right\Vert_{W^{1,\infty}}  }(1+\left\Vert B\right\Vert_{\infty}\left\Vert \eta\right\Vert_{1})\left\Vert u^{0}\right\Vert_{W^{1,\infty}}.\label{Lipschitz bound on u}      \end{align}
The bound in \eqref{Lipschitz bound on u} depends only on $T,\left\Vert B\right\Vert_{W^{1,\infty}},\left\Vert u^{0}\right\Vert_{W^{1,\infty}_{x}}$ and $\left\Vert \eta\right\Vert_{\mathrm{BV}} $. Therefore by a standard approximation argument it is straightforward to deduce that the unique solution to \eqref{nonlocal Cauhy problem simplification} is in  $W^{1,\infty}([0,T]\times \mathbb{R}^{d})$ provided $\eta\in \mathrm{BV}(\mathbb{R}^{d})$.
\end{proof}
The following elementary application of the Banach--Alaoglu theorem will be used frequently in the sequel and is therefore encapsulated in the form of a lemma. 
\begin{lem} \label{BanachAl Lemma}
Let $\{u^{\delta}\}_{\delta>0}\subset L^{\infty}([0,T];H^{-s}(\mathbb{R}^{d}))$ and suppose there is some $C>0$ such that $\left\Vert u^{\delta}\right\Vert_{L^{\infty}([0,T];H^{-s}(\mathbb{R}^{d}))} \leq C$ for all $\delta>0$. Then, there is some $u\in L^{\infty}([0,T];H^{-s}(\mathbb{R}^{d}))$ and a subsequence $\{u_{\delta_{k}}\}_{k=1}^{\infty}$ such that $u_{\delta_{k}}\underset{k\rightarrow \infty}{\rightharpoonup}u$ in $L^{2}([0,T];H^{-s}(\mathbb{R}^{d}))$.   
\end{lem}
\begin{proof}
By the theorem of Banach--Alaoglu there is a subsequence $\left\{u^{\delta_{{k}}}\right\}_{k=1}^{\infty}$ and some $u\in L^{2}([0,T];H^{-s}(\mathbb{R}^{d}))$ such that $u_{\delta_{k}}\underset{k\rightarrow \infty}{\rightharpoonup} u$ weakly in $L^{2}([0,T];H^{-s}(\mathbb{R}^{d}))$. 
For the same reason, given some $p\in (2,\infty)$, there exist some subsequence $\{u^{\delta_{k_{l}}}\}_{l=1}^{\infty}$ and some $\widetilde{u}\in L^{p}([0,T];H^{-s}(\mathbb{R}^{d}))$ such that $u^{\delta_{k_{l}}}\underset{l\rightarrow \infty}{\rightharpoonup}\widetilde{u}$ in $L^{p}([0,T];H^{-s}(\mathbb{R}^{d}))$. Moreover, there holds the estimate 
\begin{align*}
\left\Vert \widetilde{u}\right\Vert_{L^{p}([0,T];H^{-s}(\mathbb{R}^{d}))}\leq C.      
\end{align*}
Since weak convergence in $L^{p}([0,T];H^{-s}(\mathbb{R}^{d}))$ implies weak convergence in $L^{2}([0,T];H^{-s}(\mathbb{R}^{d}))$ when $p>2$, it follows that $u^{\delta_{k_{l}}}\underset{l\rightarrow \infty}{\rightharpoonup} \widetilde{u}$ weakly in $L^{2}([0,T];H^{-s}(\mathbb{R}^{d}))$. Therefore by uniqueness, it must be that $\widetilde{u}=u$. We conclude that for all $p\in (2,\infty)$ it holds that 
\begin{align*}
\left\Vert u\right\Vert_{L^{p}([0,T];H^{-s}(\mathbb{R}^{d}))}\leq C.      
\end{align*}
Letting $p\rightarrow
 \infty$ we conclude that $\left\Vert u\right\Vert_{L^{\infty}([0,T];H^{-s}(\mathbb{R}^{d}))}\leq C$ so that $u\in L^{\infty}([0,T];H^{-s}(\mathbb{R}^{d}))$. 
\end{proof}
\begin{thm}
\textup{(\textbf{Distributional solutions of \eqref{nonlocal Cauhy problem simplification}})}. \\
Suppose that: 
\begin{itemize}
    \item $u^{0}\in H^{-s}(\mathbb{R}^{d})\cap \mathcal{M}(\mathbb{R}^{d})$. 
    \item $B\in W^{s,\infty}(\mathbb{R}^{d})\cap C^{\infty}(\mathbb{R}^{d})$.
    \item $\eta\in \mathrm{BV}(\mathbb{R}^{d})$. 
\end{itemize}
Then, for any $T>0$ there exists a unique distributional solution $u\in L^{\infty}([0,T];H^{-s}(\mathbb{R}^{d}))$ with initial data $u^{0}$ to \eqref{nonlocal Cauhy problem simplification} (in the sense of Definition \ref{Defintion measure valued for nonlocal}). \label{measure valued solutions existence and uniqueness}
\end{thm}
\begin{proof}
\textbf{Existence}. Let $\{\chi_{\delta}\}_{\delta>0}$ be a family of standard mollifiers and consider the regularized problem 
\begin{align}
\partial_{t}u^{\delta}+B\cdot \nabla_{x}(\eta\ast u^{\delta})=0, \ u^{\delta}(0,\cdot)=\chi_{\delta}\ast u^{0}.\label{regularized}     
\end{align}
By Theorem
\ref{wellposedness classical sol} there exist a unique solution $u^{\delta}\in W^{1,\infty}([0,T]\times \mathbb{R}^{d})$ for \eqref{regularized}. 
Multiplying \eqref{regularized} by $\mathrm{sgn}(u^{\delta})$ and integrating we obtain  
\begin{align*}
\frac{\dd}{\dd t}\left\Vert u^{\delta}(t,\cdot)\right\Vert_{1}&=-\int_{\mathbb{R}^{d}}\mathrm{sgn}(u^{\delta})B\cdot \nabla_{x}(\eta\ast u^{\delta})\ \dd x\\
&\leq \left\Vert B\right\Vert_{\infty}\left\Vert (\nabla_{x}\eta \ast u^{\delta})(t,\cdot)\right\Vert_{1}\leq \left\Vert B\right\Vert_{\infty}\left\Vert \eta\right\Vert_{\mathrm{TV}}\left\Vert u^{\delta}(t,\cdot)\right\Vert_{1}.        
\end{align*}
Therefore we conclude the estimate 
\begin{align}
\left\Vert u^{\delta}(t,\cdot)\right\Vert_{1}\leq e^{t\left\Vert B\right\Vert_{\infty}\left\Vert \eta\right\Vert_{\mathrm{TV}} }\left\Vert \chi_{\delta}\ast u^{0}\right\Vert_{1} \leq e^{t\left\Vert B\right\Vert_{\infty}\left\Vert \eta\right\Vert_{\mathrm{TV}} }\left\vert u^
{0}\right\vert.  
\label{L1 uniform in eta}      
\end{align}
Since $L^{1}(\mathbb{R}^{d})$ is embedded is in $H^{-s}(\mathbb{R}^{d})$ for $s>\frac{d}{2}$ it follows that there is a constant $C>0$, independent of $\delta$, such that  
\begin{align*}
\left\Vert u
^{\delta}\right\Vert_{L^{\infty}([0,T];H^{-s}(\mathbb{R}^
{d}))}\leq C.      
\end{align*}
By Lemma \ref{BanachAl Lemma} there is some $u\in L^{\infty}([0,T];H^{-s}(\mathbb{R}^{d}))$ such that (up to an extraction of a subsequence) $u^{\delta}\underset{\delta \rightarrow 0}{\rightharpoonup} u$ weakly in $L^{2}([0,T];H^{-s}(\mathbb{R}^{d}
))$. 
We claim that $u$ satisfies \eqref{weak formulation for nonlocal}. Testing $u^{\delta}$ against a test function $\varphi\in C^{\infty}_{0}((-a,T)\times \mathbb{R}^{d})$ and integrating by parts we get 
\begin{align}
&\int_{0}^{T}\int_{\mathbb{R}^{d}}\partial_{\tau}\varphi(\tau,x) u^{\delta}(\tau,x)\ \dd x\dd \tau+\int_{\mathbb{R}^{d}}\varphi(0,x)\  \chi_{\delta}\ast u^{0}(x)\ \dd x \notag\\
&+\int_{0}^{T}\int_{\mathbb{R}^{d}}\mathrm{div}_{x}(\overline{\eta}\ast (B\varphi))u^{\delta}(\tau,x)\ \dd x\dd \tau=0. \label{testing udelta}   
\end{align}
Thanks to the assumption $B\in W^{s,\infty}$ we have the elementary estimate 
\begin{align*}
\left\Vert  \mathrm{div}_{x}(\overline{\eta}\ast (B\varphi)) \right\Vert_{L^{\infty}([0,T];H^{s}(\mathbb{R}^{d}))}\lesssim 1.       
\end{align*}
Therefore we may pass to the limit as $\delta \rightarrow 0$ in \eqref{testing udelta} and deduce that $u$ satisfies 
\begin{align*}
\int_{0}^{T}\langle \partial_{\tau}\varphi,u\rangle+\langle \mathrm{div}_{x}(\overline{\eta}\ast (B\varphi)),u\rangle\ \dd \tau+ \langle \varphi(0,\cdot ),u^{0}\rangle=0.         
\end{align*}
\textbf{Uniqueness}. By linearity, it suffices to show that if $u$ is a solution to \eqref{nonlocal Cauhy problem simplification} with initial data $u^{0}\equiv 0$ then $u(t,\cdot)\equiv 0$ for all $t\in [0,T]$. So assume $u^{0}\equiv0$ and let $\chi\in C^{\infty}_{0}((0,T))$ and $\varphi\in H^{s}(\mathbb{R}^{d})$. 
Then, invoking \eqref{weak formulation for nonlocal} with $\chi\varphi$ as a test function yields 
\begin{align*}
\int_{0}^{T}\partial_{\tau}\chi(\tau)\langle \varphi,u(\tau,\cdot)\rangle\ \dd \tau=-\int_{0}^{T}\chi(\tau)\langle \mathrm{div}_{x}(\overline{\eta}\ast (B\varphi)), u(\tau,\cdot) \rangle\ \dd \tau.  \end{align*}
Note that 
\begin{align*}
\langle \mathrm{div}_{x}(\overline{\eta}\ast (B\varphi)), u(\tau,\cdot) \rangle\leq \left\Vert \eta\right\Vert_{\mathrm{TV}}\left\Vert B\varphi\right\Vert_{H^{s}}  \left\Vert u\right\Vert_{L^{\infty}([0,T];H
^{-s}(\mathbb{R}^{d}))},  
\end{align*}
so that $\tau\mapsto \langle \mathrm{div}_{x}(\overline{\eta}\ast (B\varphi)), u(\tau,\cdot) \rangle \in L^{\infty}([0,T])$. 
 This proves that $\tau\mapsto \langle \varphi,u(\tau,\cdot)\rangle$ is Lipschitz with weak derivative given by 
\begin{align*}
\frac{\dd}{\dd \tau}\langle \varphi,u(\tau,\cdot)\rangle= \langle \mathrm{div}_{x}(\overline{\eta}\ast (B\varphi)),u(\tau,\cdot)\rangle.     
\end{align*}
Integrating in time we get 
\begin{align*}
\langle \varphi,u(t,\cdot)\rangle&=\int_{0}^{t}\langle \mathrm{div}_{x}(\overline{\eta}\ast (B\varphi)), u(\tau,\cdot)\rangle\ \dd \tau\leq \int_{0}^{t}\left\Vert \overline{\eta}\right\Vert_{\mathrm{TV}}\left\Vert B\varphi\right\Vert_{H^{s}}\left\Vert u(\tau,\cdot)\right\Vert_{H^{-s}}\ \dd \tau\\
&\lesssim \left\Vert \overline{\eta}\right\Vert_{\mathrm{TV}}\left\Vert B\right\Vert_{W^{s,\infty}}\left\Vert \varphi\right\Vert_{H^{s}}\int_{0}^{t}\left\Vert u(\tau,\cdot)\right\Vert_{H^{-s}}\ \dd \tau.             
\end{align*}
Maximizing over all $\varphi\in H^{s}(\mathbb{R}^{d})$ with $\left\Vert \varphi\right\Vert_{H^{s}}\leq 1$ we get 
\begin{align*}
\left\Vert u(t,\cdot)\right\Vert_{H^{-s}}\lesssim \int_{0}^{t}\left\Vert u(\tau,\cdot)\right\Vert_{H^{-s}}\ \dd \tau.      
\end{align*}
By Gr\"onwall's lemma it follows that $u\equiv 0$, as desired. 
\end{proof}
\begin{rem}
An examination of the existence proof in Theorem \ref{nonlocal Cauhy problem simplification} reveals that distributional solutions of \eqref{nonlocal Cauhy problem simplification} can be approximated through solutions of a regularized version thereof. In more detail, suppose that $u^{0}\in H^{-s}(\mathbb{R}^{d})\cap \mathcal{M}(\mathbb{R}^{d})$ and that $\{\chi_{\delta}\}_{\delta >0}$ is a family of standard mollifiers and consider the family of solutions $\{u^{\delta}\}_{\delta>0}$ of \eqref{regularized} with initial data $\chi_{\delta }\ast u^{0}$ and the unique distributional solution $u$ of \eqref{nonlocal Cauhy problem simplification} with initial data $u^{0}$. Then it holds that   $u^{\delta}\underset{\delta \rightarrow 0}{\rightharpoonup} u$ in $L^{2}([0,T];H^{-s}
(\mathbb{R}^{d}))$. \label{Remark about approximating measure valued solutions}     
\end{rem}
We proceed by proving well-posedness for the case $B(x,v)=(v,\mathbf{0})$. Note that the previous argument is not directly applicable, since we relied on the boundedness of $B$. The general strategy is to Fourier-transform the equation -- leading to a linear transport equation with source, and then show that the latter is solvable. We will need the following general Duhamel principle for initial data and source term which are assumed to be complex valued. The proof includes only minor modifications in comparison to the classical Duhamel principle, but is nevertheless included for completeness.     
\begin{lem}\label{Duhamel principle}
Suppose that: 
\begin{itemize}
    \item $W^{0}\in C^{\infty}(\mathbb{R}^{d};\mathbb{C})$.
    \item $A \in \mathrm{Lip}(\mathbb{R}^{d};\mathbb{R}^{d})\cap C^{\infty}(\mathbb{R}^{d};\mathbb{R}^{d})$.
    \item $S\in C^{\infty}(\mathbb{R}^{d};\mathbb{C})$ . 
\end{itemize}
\begin{enumerate}
    \item There exist a solution $\widetilde{W}\in W^{1,\infty}_{\mathrm{loc}}(\mathbb{R}_{+}\times \mathbb{R}^{d})$ to the  Cauchy problem 
\begin{align}
\partial_{t}\widetilde{W}-A\cdot \nabla_{x}\widetilde{W}=S+A\cdot \nabla_{x}W^{\mathrm{0}}, \ \widetilde{W}(0,\cdot)=0. \label{W tilde eq statement} 
\end{align}
\item If $\widetilde{W}$ is a solution of \eqref{W tilde eq statement} then   $W=W^{0}+\widetilde{W}$ is a solution of the Cauchy problem 
\begin{align}
\partial_{t}W-A\cdot \nabla_{x}W=S,\ W(0,\cdot)=W^{0}. \label{transport with source}    
\end{align}
\item  If $W_{1},W_{2}\in W^
{1,\infty}_{\mathrm{loc}}(\mathbb{R}_{+}\times \mathbb{R}^{d})$ are solutions to \eqref{transport with source} with the same initial data $W^{0}\in C^{\infty}(\mathbb{R}^{d};\mathbb{C})\cap L^{\infty}(\mathbb{R}^{d};\mathbb{C})$ and $S\in C^{\infty}(\mathbb{R}^{d};\mathbb{C})\cap L^{\infty}(\mathbb{R}^{d};\mathbb{C})$ then $W_{1}\equiv W_{2}$. 
\end{enumerate}
\end{lem}
\begin{proof}
\textbf{Step 1. Duhamel formula for $W^{0}=0$.} We consider the case where $W^{0}=0$. For any $s\in \mathbb{R}_{+}$ consider the IVP 
\begin{align}
\partial_{t}V_{s}-A\cdot \nabla_{x}V_{s}=0, \ V_{s}(s,\cdot)=S. \label{linear transport duhamel}    
\end{align}
By assumption $A\in \mathrm{Lip}(\mathbb{R}^{d};\mathbb{R}^{d})$, and therefore there exist a unique  solution $X(s,t,x)$ to the  ODE 
\begin{align}
\frac{\dd}{\dd t}X(s,t,x)=A(X(s,t,x)),\ X(s,s,x)=x.\label{Flow ODE}
\end{align} 
Moreover, we have that  $V_{s}(t,x)=X_{\ast}S$ is a solution to \eqref{linear transport duhamel}. We claim that 
\begin{align*}
W(t,x)\coloneqq \int_{0}^{t}V_{s}(t,x) \ \dd s     
\end{align*} 
is a solution to \eqref{W tilde eq statement}. By direct calculation it holds that   
\begin{align*}
\partial_{t}\left(\int_{0}^{t}V_{s}(t,x)\ \dd s\right)=V_{t}(t,x)+\int_{0}^{t}\partial_{t}V_{s}(t,x)\ \dd s= V_{t}(t,x)+\int_{0}^{t} A(x)\cdot \nabla_{x}V_{s}(t,x)\ \dd s 
\end{align*}   
and hence 
\begin{align*}
 \partial_{t}W(t,x)-A(x)\cdot \nabla_{x}W(t,x)=V_{t}(t,x)=S(x)   
\end{align*}
where in the last equation we used that  $X(t,t,x)=x$. \\
\textbf{Step 2}. We consider now the case of general $W^{0}\in C^{\infty}(\mathbb{R}^{d};\mathbb{C})$. By the previous step there exist a  solution $\widetilde{W}$  to the Cauchy problem 
\begin{align}
\partial_{t}\widetilde{W}-A\cdot \nabla_{x}\widetilde{W}=S+A\cdot \nabla_{x}W^{0}, \ \widetilde{W}(0,\cdot)=0. \label{W tilde eq}  
\end{align}
Define $W(t,x)\coloneqq W^{0}+\widetilde{W}$.  Then we have 
\begin{align*}
\partial_{t}W-A\cdot \nabla_{x}W&=\partial_{t}\widetilde{W}-A\cdot \nabla_{x}(\widetilde{W}+W^{0})\\
&=S+A\cdot \nabla_{x}(\widetilde{W}+W^{0})-A\cdot \nabla_{x}(\widetilde{W}+W^{0})=S.     
\end{align*}
It follows that $W=W^{0}+\widetilde{W}$ solves \eqref{transport with source} where $\widetilde{W}$ is a solution of \eqref{W tilde eq statement}.\\ 
\textbf{Step 3. Uniqueness.} If $W_{1},W_{2}$ are two solutions of \eqref{transport with source} with the same initial data $W^{0}$ then the difference $W_{1}-W_{2}$ satisfies  \begin{align}
\partial_{t}(W_
{1}-W_{2})-A\cdot \nabla_{x}(W_{1}-W_{2})=0, \ (W_{1}-W_{2})(0,\cdot)=0. \label{equation for difference W12}    
\end{align}
Since $A$ is real valued, by taking the real and imaginary parts of  \eqref{equation for difference W12} we conclude that 
\begin{align}
\partial_{t}\Re(W_{1}-W_{2})-A\cdot \nabla_{x}\Re(W_{1}-W_{2})=0 \label{real valued part} 
\end{align}
and 
\begin{align}
\partial_{t}\Im(W_{1}-W_{2})-A\cdot  \nabla_{x}\Im(W_{1}-W_{2})=0.\label{imaginary part}     
\end{align}
Multiplying \eqref{real valued part}-\eqref{imaginary part} by $\mathrm{sgn}(\Re(W_{1}-W_{2}))$ and $\mathrm{sgn}(\Im(W_{1}-W_{2}))$ respectively we get 
\begin{align}
\partial_{t}\left\vert \Re(W_{1}-W_{2})\right\vert -A\cdot \nabla_{x}\left\vert \Re(W_{1}-W_{2})  \right\vert=0 
\label{real absolute}
\end{align}
and 
\begin{align}
\partial_{t}\left\vert \Im(W_{1}-W_{2})\right\vert-A\cdot \nabla_{x}\left\vert \Im(W_{1}-W_{2})\right\vert=0. \label{imaginary absolute}      
\end{align}
Evaluating \eqref{real absolute}-\eqref{imaginary absolute} at a maximum point of $\left\vert \Re(W_{1}-W_{2})\right\vert$ and $\left\vert \Im(W_{1}-W_{2})\right\vert$ respectively we conclude that 
\begin{align*}
\frac{\dd}{\dd t}\left\Vert \Re(W_{1}-W_{2})(t,\cdot)\right\Vert_{\infty}=\frac{\dd}{\dd t}\left\Vert \Im(W_{1}-W_{2})(t,\cdot)\right\Vert_{\infty}=0,     
\end{align*}
so that $\Re(W_{1})=\Re(W_{2}), \ \Im(W_{1})=\Im(W_{2})$. Thus $W_{1}\equiv W_{2}$, proving uniqueness. 
\end{proof}
\begin{rem}
Note that if $A$ is a \textbf{complex} valued vector field then \eqref{transport with source} becomes a \textbf{system} of linear PDEs whose well-posedness theory does not seem direct. Indeed, if $A$ is complex valued then splitting the system to real and imaginary parts we get 
\begin{align*}
\begin{cases}
\begin{array}{lc}
\partial_{t}W_{R}-A_{R}\cdot \nabla_{x}W_{R}+A_{I}\cdot \nabla_{x}W_{I}=S_{R}\\
\partial_{t}W_{I}-A_{R}\cdot \nabla_{x}W_{I}-A_{I}\cdot\nabla_{x}W_{R}=S_{I}.  
\end{array}\end{cases}
\end{align*}
It is not clear whether this system formally conserves $L^{p}$ norms and therefore we did not attempt to study this system at this level of generality, which is why  we require $A$ to be real valued. This is essentially also the reason we impose the requirement that $\eta$ is even (so that its Fourier transform is real valued).   
\end{rem}
We now apply Lemma \ref{Duhamel principle} in order to study the well-posedness theory of the Fourier-transformed equation. We designate by $L^{2}_{-s}$   the weighted $L^{2}$ space defined by $L^{2}_{-s}\coloneqq L^{2}(w_{-s}(\xi,\zeta) \dd \xi\dd \zeta)$ where $w_{-s}(\xi,\zeta)\coloneqq (1+\left\vert \xi\right\vert^{2}+\left\vert \zeta\right\vert^{2} )^{-s}$.  
\begin{cor}
Suppose that: 
\begin{itemize}
    \item $U^{0}\in C^{\infty}(\mathbb{R}^{d}_{\xi}\times \mathbb{R}_{\zeta}^{d};\mathbb{C})$ and $U^{0}(\xi,\zeta)\neq 0$ for all $(\xi,\zeta)$. 
    \item $A\in  \mathrm{Lip}(\mathbb{R}^{d}_{\xi}\times \mathbb{R}^{d}_{\zeta};\mathbb{R}^{d}\times \mathbb{R}^{d})\cap C^{\infty}(\mathbb{R}^{d}_{\xi}\times \mathbb{R}^{d}_{\zeta};\mathbb{R}^{d}\times \mathbb{R}^{d}) $.  
    \item $C\in C^{\infty}(\mathbb{R}^{d}_{\xi}\times \mathbb{R}^{d}_{\zeta};\mathbb{C})$. 
\end{itemize}
\begin{enumerate}
    \item  There exist a solution $U\in W^{1,\infty}_{\mathrm{loc}}(\mathbb{R}_{+}\times \mathbb{R}^{d}_{\xi}\times \mathbb{R}^{d}_{\zeta})$ to the Cauchy problem 
\begin{align}
\partial_{t}U(\xi,\zeta)-A(\xi,\zeta)\cdot \nabla_{\zeta}U(\xi,\zeta)=C(\xi,\zeta)U(\xi,\zeta),\ U(0,\cdot)=U^{\mathrm{0}}.  \label{corollary cauchy problem}    
\end{align}
\item Suppose in addition that: 
\begin{itemize}
\item $C\in L^{\infty}(\mathbb{R}^{d}_{\xi}\times \mathbb{R}^{d}_{\zeta})$ is real valued.
\item $\left\vert A(\xi,\zeta)\right\vert\leq \left\vert \xi\right\vert$ and $\mathrm{div}_{\zeta}A\in L^{\infty}(\mathbb{R}^{d}_{\xi}\times \mathbb{R}^{d}_{\zeta})$. 
\item $U^
{0}\in L^{2}_{-s}$. 
\end{itemize}
If $U$ is a solution to \eqref{corollary cauchy problem} with initial data $U^{0}$ then for all $t\in \mathbb{R}_{+}$ the following estimate hold for some $C=C(\left\Vert \mathrm{div}_{\zeta}A\right\Vert_{\infty},s,\left\Vert C\right\Vert_{\infty})$:  
\begin{align}
\left\Vert U(t,\cdot)\right\Vert_{L^{2}_{-s}} \leq e^{Ct}\left\Vert U^{0}\right\Vert_{L^{2}_{-s}}.      \label{L2s of U} 
\end{align}
\end{enumerate}
\label{linear transport well posed}
\end{cor}
\begin{proof}
\textbf{Step 1. Existence.} Since $U^
{0}(\xi,\zeta)\neq 0$ for all $(\xi,\zeta)$ we have $\nabla_{\zeta}\log U^{0}\in C^{\infty}(\mathbb{R}^{d}_{\xi}\times \mathbb{R}^{d}_{\zeta};\mathbb{C})$. Let $\widetilde{W}\in W^{1,\infty}_{\mathrm{loc}}(\mathbb{R}_{+}\times \mathbb{R}^{d}_{\xi}\times \mathbb{R}^{d}_{\zeta})$ be a solution of 
\begin{align*}
\partial_{t}\widetilde{W}-A\cdot \nabla_{\zeta}\widetilde{W}=C+A\cdot\nabla_{\zeta}\log(U^{0}), \ \widetilde{W}(0,\cdot)=0      
\end{align*}
as guaranteed via Lemma \ref{Duhamel principle}.  
Define  $U(t,\xi,\zeta)\coloneqq U^{0}(\xi,\zeta)e^{\widetilde{W}(t,\xi,\zeta)}$. We compute: 
\begin{align*}
\partial_{t}U-A\cdot\nabla_{\zeta}U&=U^{0}e^{\widetilde{W}}\partial_{t}\widetilde{W}-A\cdot \nabla_{\zeta}(U^{0}e^{\widetilde{W}})\\
&=U^{0}e^{\widetilde{W}}\left(\partial_{t}\widetilde{W}-A\cdot \nabla_{\zeta}\widetilde{W}\right)-A\cdot \nabla_{\zeta}U^{0}e^{\widetilde{W}}\\
&= U^{0}e^{\widetilde{W}}(C+A\cdot\nabla_{\zeta}\log(U^{0}))-A\cdot \nabla_{\zeta}U^{0}e^{\widetilde{W}}\\
&=U^{0}e^{\widetilde{W}}C=UC. 
\end{align*}
\textbf{Step 2. The estimate \eqref{L2s of U} .} 
To establish \eqref{L2s of U}  we assume first that $U^{0}$ is real valued, so that $U(t,\cdot)$ is also real valued. 
Multiplying \eqref{corollary cauchy problem} by $w_{-s}(\xi,\zeta)U$ we get 
\begin{align}
\partial_{t}(w_{-s}(\xi,\zeta)\left\vert U\right\vert^{2})&-\frac{1}{2}A\cdot \nabla_{\zeta}(w_{-s}(\xi,\zeta)\left\vert U\right\vert^
{2}) \notag\\
&+\frac{1}{2}A\cdot\nabla_{\zeta}(w_{-s}(\xi,\zeta))\left\vert U\right\vert^{2}=C(\xi,\zeta)w_{-s}(\xi,\zeta)\left\vert U\right\vert^{2}.\label{eq for weighted U2}       
\end{align}
Therefore, integrating in $(\xi,\zeta)$ we obtain 
\begin{align*}
\frac{\dd}{\dd t}\left\Vert U(t,\cdot)\right\Vert_{L^{2}_{-s}}^{2} =&-\frac{1}{2}\int_{\mathbb{R}^{d}\times \mathbb{R}^{d}}\mathrm{div}_{\zeta}(A)(\xi,\zeta)w_{-s}(\xi,\zeta)\left\vert U\right\vert^{2}(t,\xi,\zeta)\ \dd \xi\dd \zeta\\
&-\frac{1}{2}\int_{\mathbb{R}^{d}\times \mathbb{R}^{d}}A\cdot \nabla_{\zeta}w_{-s}(\xi,\zeta)\left\vert U \right\vert^{2}(t,\xi,\zeta) \ \dd \xi \dd \zeta
\\
&+\int_{\mathbb{R}^{d}\times \mathbb{R}^{d}}C(\xi,\zeta)w_{-s}(\xi,\zeta)\left\vert U\right\vert^{2}(t,\xi,\zeta)\ \dd \xi\dd \zeta.   
\end{align*}
By assumption $\left\vert A(\xi,\zeta)\right\vert\leq \left\vert \xi\right\vert$ and therefore  
\begin{align*}
\left\vert A\cdot \nabla_{\zeta}w_{-s}(\xi,\zeta)\right\vert =2s\left\vert A\cdot \frac{\zeta}{(1+\left\vert \xi\right\vert^{2}+\left\vert \zeta\right\vert^{2})^{s}}\right\vert \leq 2s\left\vert \frac{\xi\cdot \zeta}{(1+\left\vert \xi\right\vert^{2}+\left\vert \zeta\right\vert^{2})^{s+1} }\right\vert\leq sw_{-s}(\xi,\zeta).      
\end{align*}
Consequently, it follows that 
\begin{align}
\frac{\dd}{\dd t}\left\Vert U(t,\cdot)\right\Vert_{L^{2}_{-s}}^{2}&\leq \left\Vert \mathrm{div}_{\zeta}A\right\Vert_{\infty} \left\Vert U(t,\cdot)\right\Vert_{L^{2}_{-s}}^{2} +s\left\Vert U(t,\cdot)\right\Vert_{L^{2}_{-s}}^{2} 
+\left\Vert C\right\Vert_{\infty}\left\Vert U(t,\cdot)\right\Vert_{L^{2}_{-s}}^
{2} \notag\\
&\leq (\left\Vert \mathrm{div}_{\zeta}A\right\Vert_{\infty} +s+\left\Vert C\right\Vert_{\infty})\left\Vert U(t,\cdot)\right\Vert_{L^{2}_{-s}}^{2}. \label{first gronwall type} \end{align}
We therefore conclude that 
\begin{align*}
\left\Vert U(t,\cdot)\right\Vert_{L^{2}_{-s}}\leq e^{(\left\Vert \mathrm{div}_{\zeta}A\right\Vert_{\infty}+s+\left\Vert C\right\Vert_{\infty})t}\left\Vert U^{0}\right\Vert_{L^{2}_{-s}}.       
\end{align*}
\end{proof}

\begin{thm}
\textup{(\textbf{Distributional solutions for nonlocal free transport})}\label{Measure valued solution to Liouville eq}\\
Suppose that:  
\begin{itemize}
    \item $u^{0}\in H^{-s}(
    \mathbb{R}^{d}_{x}\times \mathbb{R}^{d}_{v})\cap \mathcal{M}_{c}(\mathbb{R}^{d}_{x}\times \mathbb{R}^{d}_{v})$ for some $s>\frac{d}{2}$. Moreover  $\widehat{u}^{0}(\xi,\zeta)\neq 0$ for all $(\xi,\zeta)$. 
    \item $\eta\in \mathcal{S}(\mathbb{R}^{d}_{x}\times \mathbb{R}^{d}_{v})$ is even , i.e. $\eta(x,v)=\eta(-x,-v)$.  
\end{itemize}
\begin{enumerate}
    \item There exist a  distributional solution $u\in L^{\infty}([0,T];H^{-s}(\mathbb{R}^{d}\times \mathbb{R}^{d}))$ (in the sense of Definition \ref{Defintion measure valued for nonlocal}) to the  Cauchy problem
\begin{align}
\partial_{t}u+v\cdot \nabla_{x}(\eta\ast u)=0, \ u(0,\cdot)=u^{0}.  \label{Liouville eq}  \end{align}
\item If $u,\widetilde{u}\in L^{\infty}([0,T];H^{-s}(\mathbb{R}^{d}\times \mathbb{R}^{d}))$ are distributional solutions of \eqref{Liouville eq} with the same initial data $u^{0}$ then $u\equiv \widetilde{ u}$.
\end{enumerate}
\end{thm}
\begin{proof}
The idea underpinning the proof is to consider the equation obtained by taking the Fourier transform of \eqref{Liouville eq}, which has the same form as \eqref{corollary cauchy problem}. \\
\textbf{Step 1. Schwartz initial data.} Let $u^{0}\in \mathcal{S}(\mathbb{R}^{d}_{x}\times \mathbb{R}^{d}_{v})$ be such that $\widehat{u}^{0}$ is everywhere non-vanishing. We consider the following linear Cauchy problem 
\begin{align}
\partial_{t}U(\xi,\zeta)-\widehat{\eta}(\xi,\zeta)\xi\cdot \nabla_{\zeta}U(\xi,\zeta)=\xi\cdot \nabla_{\zeta}\widehat{\eta} (\xi,\zeta) U(\xi,\zeta),  \ U(0,\cdot)=\widehat{u}^{0}. \label{equation for Fourier transform} 
\end{align}
Since $\eta$ is even and $\eta \in \mathcal{S}(\mathbb{R}^{d}_{x}\times \mathbb{R}^{d}_{v})$ it follows that $\widehat{\eta}$ is real valued and that  $\widehat{\eta}\in \mathcal{S}(\mathbb{R}^{d}_{\xi}\times \mathbb{R}^{d}
_{\zeta})$. Therefore we can apply Corollary \ref{linear transport well posed} with $A(\xi,\zeta)=\widehat{\eta}(\xi,\zeta)(\mathbf{0},\xi)$ and $C(\xi,\zeta)=\xi\cdot \nabla_{\zeta}\widehat{\eta}(\xi,\zeta)$ in order to deduce the existence of a solution $U$ to \eqref{equation for Fourier transform}. Set  $u(t,x,v)\coloneqq \mathcal{F}^{-1}_{(\xi,\zeta)}(U(t,\cdot))(x,v)$ and notice  the identity 
\begin{align*}
\mathcal{F}^{-1}_{(\xi,\zeta)}(\widehat{\eta}(\xi,\zeta)\xi\cdot \nabla_{\zeta}U(\xi,\zeta)+\xi\cdot \nabla_{\zeta}\widehat{\eta}(\xi,\zeta)U(\xi,\zeta))(x,v)=-v\cdot \nabla_{x}(\eta\ast u)(x,v). \end{align*}
Therefore, we infer that 
\begin{align*}
\partial_{t}u+v\cdot \nabla_{x}(\eta\ast u)=0,\ u(0,\cdot)=u^{0}.     
\end{align*}
\textbf{Step 2. General initial data}. Now consider initial data $u
^{0}\in H^{-s}(\mathbb{R}^{d}_{x}\times \mathbb{R}^{d}_{v})\cap \mathcal{M}_{c}(\mathbb{R}^{d}_{x}\times \mathbb{R}^{d}_{v})$ such that $\widehat{u}^{0}$ is everywhere non-vanishing. Fix a family of standard mollifiers $\{\chi_{\delta}\}_{\delta>0}\subset \mathcal{S}(\mathbb{R}^{d}_{x}\times \mathbb{R}^{d}_{v})$ such that $\widehat{\chi}_{\delta}$ is everywhere non-vanishing. Consequently, $\chi_{\delta}\ast u^{0}\in \mathcal{S}(\mathbb{R}^{d}_{x}\times \mathbb{R}^{d}_{v})$ and  $\widehat{\chi_{\delta}\ast u_{0}}=\widehat{\chi}_{\delta}\widehat{u}^{0}$ is non-vanishing. For each $\delta>0$ let  $u^{\delta}$ be a solution of the regularized problem 
\begin{align}
\partial_{t}u^{\delta}+v\cdot\nabla_{x}(\eta\ast u^{\delta})=0, \ u^{\delta}(0,\cdot)=\chi_{\delta}\ast u^{0},\label{Regulazrized system liouville}    
\end{align}
whose existence is guaranteed thanks to step 1. Taking the Fourier transform of \eqref{Regulazrized system liouville} and setting $U^{\delta}\coloneqq \widehat{u}^{\delta}$ we get 
\begin{align*}
\partial_{t}U^{\delta}-\widehat{\eta}(\xi,\zeta)\xi\cdot \nabla_{\zeta}U^{\delta}=\xi\cdot \nabla_{\zeta}\widehat{\eta}(\xi,\zeta)U^{\delta}(\xi,\zeta), \ U^{\delta}(0,\cdot)\coloneqq   \widehat{\chi}(\delta\cdot )\widehat{u}^{0}.    
\end{align*}
By Lemma \ref{linear transport well posed} there is some $C>0$ (independent of $\delta$) such that 
\begin{align*}
\left\Vert u^{\delta}\right\Vert_{L^{\infty}([0,T];H^{-s})} =\underset{t\in [0,T]}{\sup}\left\Vert U^{\delta}(t,\cdot)\right\Vert_{L^{2}_{-s}
}\leq C\left\Vert U^{\delta}(0,\cdot)\right\Vert_{L^{2}_{-s}}\leq C\left\Vert u^{0}\right\Vert_{H^{-s}}.   
\end{align*}
By Lemma \ref{BanachAl Lemma} there exists a subsequence $u^{\delta}$ (not relabeled) and some $u\in L^{\infty}([0,T];H^{-s}(\mathbb{R}^{d}_{x}\times \mathbb{R}^{d}_{v}))$ such that $u^{\delta}\underset{\delta \rightarrow 0}{\rightharpoonup}u$ weakly in $L^{2}([0,T];H^{-s})$. For any $\varphi \in C^{\infty}_{0}((-a,T)\times \mathbb{R}^{d}_{x}\times \mathbb{R}^{d}_{v})$ it holds that  
\begin{align}
\int_{0}^{T}\langle \partial_{\tau}\varphi(\tau,\cdot),u^{\delta}(\tau,\cdot)\rangle+\langle \mathrm{div}_{x}(\eta\ast (v\varphi)),u^{\delta}(\tau,\cdot)\rangle \ \dd \tau+\langle \varphi(0,\cdot),\chi_{\delta}\ast u^{0}\rangle=0.     
\label{tested equation}
\end{align}
Since $\eta\ast \mathrm{div}_{x}(v\varphi)\in L^{2}([0,T];H^{s}(\mathbb{R}^{d}_{x}\times \mathbb{R}^{d}_{v}))$ we can pass to the limit as $\delta \rightarrow 0$ in \eqref{tested equation}  
and conclude that $u$ is a distributional solution of \eqref{Liouville eq}.\\ 
\textbf{Step 3. Uniqueness.} 
It suffices to show that if  
$u$ is a measure valued solution with initial data $u(0,\cdot)\equiv0$ then $u\equiv 0$ for all times.  For all $\varphi\in C^{\infty}_{0}((-a,T);\mathcal{S}(\mathbb{R}^{d}\times \mathbb{R}^{d};\mathbb{C}))$ it holds that 
\begin{align*}
\int_{0}^{T}\langle \partial_{\tau}\overline{\varphi}(\tau,\cdot),u(\tau,\cdot)\rangle\ \dd \tau+\int_{0}^{T}\langle \mathrm{div}_{x}(\eta\ast (v\overline{\varphi})), u(\tau,\cdot)\rangle \ \dd \tau=0.    
\end{align*}
Passing to Fourier we get  
\begin{align*}
\int_{0}^{T}\langle \partial_{\tau}\overline{\widehat{\varphi}}(\tau,\cdot),\widehat{u}(\tau,\cdot)\rangle\ \dd \tau+\int_{0}^{T}\langle \widehat{\eta}\xi\cdot \nabla_{\zeta}\overline{\widehat{\varphi}},\widehat{u}(\tau,\cdot)\rangle\ \dd \tau=0.      
\end{align*}
Since the Fourier transform is a bijection from $\mathcal{S}(\mathbb{R}^{d}\times \mathbb{R}^{d};\mathbb{C})$ to itself it follows that for all $\psi\in C^{\infty}_{0}((-a,T);\mathcal{S}(\mathbb{R}^{d}\times \mathbb{R}^{d};\mathbb{C}))$
it holds that 
\begin{align*}
\int_{0}^{T}\langle \partial_{\tau}\psi(\tau,\cdot),\widehat{u}(\tau,\cdot)\rangle\ \dd \tau+\int_{0}^{T}\langle \widehat{\eta}\xi\cdot \nabla_{\zeta}\psi,\widehat{u}(\tau,\cdot)\rangle\ \dd \tau=0.      
\end{align*}
For any $\Psi\in  C^{\infty}_{0}((-a,T)\times \mathbb{R}^{d}\times \mathbb{R}^{d})$ we may find $\psi\in C^{\infty}_{0}((-a,T); \mathcal{S}(\mathbb{R}^{d}\times \mathbb{R}^{d}))$ such that $\partial_{\tau}\psi+\widehat{\eta}\xi\cdot \nabla_{\zeta}\psi=\Psi$ and therefore we conclude that for all $\Psi\in C^{\infty}_{0}((-a,T)\times \mathbb{R}^{d}\times \mathbb{R}^{d})$ it holds that 
\begin{align*}
\int_{0}^{T}\langle \Psi(\tau,\cdot),\widehat{u}(\tau,\cdot)\rangle\ \dd \tau=0      
\end{align*}
and so we conclude that $\widehat{u}\equiv 0$ and therefore $u\equiv 0$, as desired. 
\end{proof}
\section{Proof of Theorem \ref{Constant transport limit} and Theorem \ref{Free transport limit}}\label{Proof of main theorems sec}
In this section we continue to make extensive use of the Fourier approach in order to prove the nonlocal to local limit as stated in Theorem \ref{Constant transport limit} and Theorem \ref{Free transport limit}. We will need the following lemma regarding uniqueness of a certain distributional logistic equation.  
\begin{lem} \label{Uniqueness lemma}
Suppose that: 
\begin{itemize}
    \item $f\in \mathcal{S}(\mathbb{R}^{d};\mathbb{C})$ is such that $(\Re f+1)^{2}(\xi)+(\Im f)^{2}(\xi)> 0$ for a.e. $\xi\in \mathbb{R}^{d}$. 
    \item $U^{0}\in L^{2}_{\mathrm{loc}}(\mathbb{R}^{d};\mathbb{C})$.
\end{itemize}
Let $U\in L^{\infty}
([0,T];L^{2}_{\mathrm{loc}}(\mathbb{R}^{d};\mathbb{C}))$ be such that 
\begin{align}
\int_{0}^{T}\langle \partial_{\tau}\psi,U\rangle+\langle \psi, fU \rangle\ \dd \tau+\langle U^{0},\psi(0,\cdot)\rangle =0\ \mbox{for all}\ \psi\in C^{\infty}_{0}((-a,T);\mathcal{S}(\mathbb{R}^{d})). 
\label{logistic distrubtional equation}
\end{align}
Then, there is a unique solution to \eqref{logistic distrubtional equation} given by   $U(t,\xi)=e^{ft}U^{0}(\xi)$. 
\end{lem}
\begin{proof}
\textbf{Step 1. Uniqueness.} First, we prove uniqueness for \eqref{logistic distrubtional equation}. By linearity it is enough to prove that if $U\in L^{\infty}([0,T];L^{2}_{\mathrm{loc}}(\mathbb{R}^{d};\mathbb{C}))$ is a solution with initial data $U^{0}\equiv 0$ then $U\equiv 0$ for all times. Given $\Phi\in C^{\infty}_{0}((-a,T);\mathcal{S}(\mathbb{R}^{d}))$ define  
\begin{align*}
\psi(\tau,x)=e^{-\tau}\int^{\tau}_{0}e^{s}\Phi(s,x)\ \dd s-e^{-\tau}\int_{0}^{T}e^{s}\Phi(s,x)\ \dd s.    
\end{align*}
By direct calculation it is verified  that 

\begin{align*}
\partial_{\tau}\psi+\psi=\Phi \ \mbox{on}\ [0,T).    
\end{align*}
Moreover it holds that $\psi(T,\cdot)\equiv 0.$ Finally, we can extend $\psi$ smoothly to $(-a,T)$ so that $\psi(t,\cdot)$ vanishes near $-a$.    
Therefore by \eqref{logistic distrubtional equation} we conclude that for all $\Phi\in C^{\infty}_{0}((-a,T);\mathcal{S}(\mathbb{R}^{d}))$ it holds that 
\begin{align*}
\int_{0}^{T}\langle \Phi,U+fU \rangle\ \dd \tau=0.     
\end{align*}
Consequently it follows that  
\begin{align}
U+fU=0 \ \mbox{for  a.e}\ (t,\xi)\in [0,T]\times \mathbb{R}^{d}.  \label{UfU}  
\end{align}
We abbreviate $U_{R}=\Re U, U_{I}=\Im U$ and similarly for $f$. Taking the real part and imaginary parts of \eqref{UfU} we arrive at the following set of equations: 
\begin{align}
\begin{cases}
\begin{array}{lc}
U_{R}-f_{I}U_{I}+f_{R}U_{R}=0\\
U_{I}+f_{I}U_{R}+f_{R}U_{I}=0. \end{array}\end{cases} \label{system of eq}
\end{align}
In matrix form, the system \eqref{system of eq} writes 
\begin{align*}
\begin{pmatrix}U_{R}\\
U_{I}
\end{pmatrix}=A\begin{pmatrix}U_{R}\\
U_{I}
\end{pmatrix}\ \mbox{where}\ A\coloneqq\begin{pmatrix}-f_{R} & f_{I}\\
-f_{I} & -f_{R}
\end{pmatrix}.     
\end{align*}
Note that the equation $Av=v$ can be rewritten as $(A-I)v=0$ and therefore its only solution is the trivial solution iff $\det(A-I)\neq 0$. Furthermore we compute that 
\begin{align*}
\det (A-I)=(f_{R}+1)^{2}+f_{I}^{2}.
\end{align*}
By assumption we have 
$(f_{R}+1)^{2}(\xi)+f_{I}^{2}(\xi)\neq 0$ for a.e. $\xi\in \mathbb{R}^{d}$.
Consequently, it follows that  $U\equiv 0$. \\
\textbf{Step 2.} It remains to verify that $U(t,\xi)=e^{tf}U^{0}(\xi)\in L^{\infty}([0,T];L^{2}_{\mathrm{loc}}(\mathbb{R}^{d};\mathbb{C}))$ is a solution to \eqref{logistic distrubtional equation}. Indeed, for any $\psi\in C^{\infty}_{0}((-a,T);\mathcal{S}(\mathbb{R}^{d}))$ we may integrate by parts to get  
\begin{align*}
&\int_{0}^{T}\langle \partial_{\tau}\psi,U\rangle+\langle \psi,fU\rangle\ \dd \tau+\langle \psi(0,\cdot),U^{0}\rangle\\
=&\int_{0}^{T}\int_{\mathbb{R}^{d}}e^{f\tau}U^{0}(\xi)\partial_{\tau}\psi +fe^{f\tau}U^{0}(\xi)\psi\ \dd \xi\dd \tau+\int_{\mathbb{R}^{d}}U^{0}(\xi)\psi(0,\xi)\ \dd \xi\\
=&-\int_{0}^{T}\int_{\mathbb{R}^{d}}fe^{f\tau}U^{0}(\xi)\psi
\ \dd\xi\dd \tau+\int_{0}^{T}\int_{\mathbb{R}^{d}}fe^{f\tau}U^{0}(\xi)\psi\ \dd \xi\dd \tau\\\
&+\int_{\mathbb{R}^{d}}U^{0}(\xi)\psi(0,\xi)\ \dd \xi-\int_{\mathbb{R}^{d}}U^{0}(\xi)\psi(0,\xi)\ \dd \xi=0.  
\end{align*}
\end{proof}
\textit{Proof of Theorem \ref{Constant transport limit}}. With no loss of generality we assume that $B\neq 0$. 
By passing to Fourier we see that for any $\varphi\in C^{\infty}_{0}((-a,T); \mathcal{S}(\mathbb{R}^{d};\mathbb{C}))$ it holds that  
\begin{align*}
\int_{0}^{T}\langle \partial_{\tau}\overline{\widehat{\varphi}},\widehat{u}_{\varepsilon}\rangle \ \dd \tau
-\int_{0}^{T}\langle \overline{\widehat{\varphi}},i\widehat{B}\ast (\xi\widehat{\eta}(-\varepsilon\cdot))\widehat{u}_{\varepsilon}\rangle \ \dd\tau+\langle \overline{\widehat{\varphi}},\widehat{u}^{0}\rangle =0, 
\end{align*}
and therefore 
\begin{align}
\int_{0}^{T}\langle \partial_{\tau}\overline{\widehat{\varphi}},\widehat{u}_{\varepsilon}\rangle\ \dd \tau-\int_{0}^{T}\langle \overline{\widehat{\varphi}},iB\cdot \xi\widehat{\eta}(-\varepsilon \cdot)\widehat{u}_{\varepsilon}\rangle\ \dd \tau+\langle \overline{\widehat{\varphi}},\widehat{u}^{0}\rangle=0.    
\end{align}
Since the Fourier transform is a bijection from $\mathcal{S}(\mathbb{R}^{d};\mathbb{C})$ to itself it follows that for all $\psi\in C^{\infty}_{0}((-a,T);\mathcal{S}( \mathbb{R}^{d}))$ it holds that 
\begin{align*}
\int_{0}^{T}\langle \partial_{\tau}\psi,\widehat{u}_{\varepsilon}\rangle-\int_{0}^{T}\langle \psi, iB\cdot \xi\widehat{\eta}(-\varepsilon \cdot)\widehat{u}_{\varepsilon}\rangle\ \dd \tau+\langle \psi(0,\cdot), \widehat{u}^{0}\rangle=0.        
\end{align*}
Note that since $B\neq 0$ the set $\{\xi\in \mathbb{R}^{d}\vert B\cdot \xi =0\}$ is of measure $0$. Thus, by the assumption on $\eta$ we have   
\begin{align*}
(1-B\cdot \xi\Im(\widehat{\eta})(\varepsilon\xi))^{2}+\left\vert B\cdot \xi\right\vert^{2}\left\vert \Re(\widehat{\eta})(\varepsilon\xi)\right\vert^{2}> 0 \ \mbox{for a.e}\ \xi.       
\end{align*}
Therefore, by Lemma \ref{Uniqueness lemma} we conclude that $\widehat{u}_{\varepsilon}(t,\xi)$ is given by the formula     
\begin{align*}
\widehat{u}_{\varepsilon}(t,\xi)=e^{-itB\cdot \xi \widehat{\eta}(-\varepsilon \xi)}\widehat{u}_{0}(\xi).     
\end{align*}
From the same considerations we conclude that  $\widehat{u}$ given by the formula   \begin{align}
\widehat{u}(t,\xi)=e^{-itB\cdot \xi}\widehat{u}^{0}(\xi). \label{Fourier of u formula}    
\end{align}
Therefore, we conclude that 
\begin{align*}
\left\Vert u_{\varepsilon}(t,\cdot)-u(t,\cdot)\right\Vert_{H^{-s}}^{2}=\int_{\mathbb{R}^{d}}\frac{\left\vert \widehat{u}_{\varepsilon}(t,\xi)-\widehat{u}(t,\xi)\right\vert^{2}}{(1+\left\vert \xi\right\vert^{2})^{s}}\ \dd \xi=\int_{\mathbb{R}^{d}}\frac{\left\vert \widehat{u}_{0}(\xi)(e^{-itB\cdot \xi}-e^{-itB\cdot \xi \widehat{\eta}(-\varepsilon \xi)})\right\vert^{2} }{(1+\left\vert \xi\right\vert^{2})^{s}}\ \dd \xi.    
\end{align*}
Now, observe that: 
\begin{itemize}
    \item $\frac{\left\vert \widehat{u}_{0}\right\vert^{2} }{(1+\left\vert \xi\right\vert^{2})^{s}}\in L^{1}(\mathbb{R}^{d})$ because  $\xi\mapsto \frac{1}{(1+\left\vert \xi\right\vert^{2})^{s}}\in L^{1}(\mathbb{R}^{d})$ whenever $s>\frac{d}{2}$ and  $\widehat{u}_{0}\in L^{\infty}(\mathbb{R}^{d})$ as a Fourier transform of a finite Radon measure. 
    \item $\underset{t\in [0,T]}{\sup}\left\vert e^{-itB\cdot \xi}-e^{-itB\cdot \xi \widehat{\eta}(\varepsilon\xi)}\right\vert \underset{\varepsilon \rightarrow 0}{\rightarrow} 0$ because $\widehat{\eta}$ is continuous and $\widehat{\eta}(0)=1$. 
\end{itemize}
Therefore, we conclude by the dominated convergence theorem that 
\begin{align*}
\underset{t\in [0,T]}{\sup}\left\Vert (u_{\varepsilon}-u)(t,\cdot)\right\Vert_{H^{-s}}\underset{\varepsilon\rightarrow 0}{\rightarrow}0.      
\end{align*}
\qed
\\\\
Before presenting the proof of Theorem \ref{Free transport limit} we need the following simple lemma. 
\begin{lem}\label{bound on negative weighted fourier}
Suppose that: 
\begin{itemize}
    \item $u^{0}\in H^{-s}(\mathbb{R}^{d}\times \mathbb{R}^{d})\cap \mathcal{M}_{c}(\mathbb{R}^{d}\times \mathbb{R}^{d})$. 
    \item $B(x,v)=(v,\mathbf{0})$.
    \item $u$ is a distributional solution to \eqref{local Cauchy} with initial data $u^{0}$. 
\end{itemize}
Then, it holds that 
\begin{align*}
\left\Vert \widehat{u}(t,\cdot)\right\Vert_{\infty}=\left\Vert \widehat{u}^{0}\right\Vert_{\infty}\ \mbox{for all}\ t\in \mathbb{R}_{+}     
\end{align*}
and 
\begin{align}
\left\Vert  \nabla_{\zeta}\widehat{u}(t,\cdot)\right\Vert_{\infty}=\left\Vert \nabla_{\zeta}\widehat{u}^{0}\right\Vert_{\infty} \ \mbox{for all}\ t\in \mathbb{R}_{+}.   \label{Linfty of weighted grad fourier}    
\end{align}
\end{lem}
\begin{proof}
Note that $\widehat{u}^{0}$ is smooth because $u^{0}\in \mathcal{M}_{c}(\mathbb{R}^{d}\times \mathbb{R}^{d})$. Therefore, the Fourier transform of $u(t,\xi,\zeta)$ is a classical solution of the equation 
\begin{align*}
\partial_{t}\widehat{u}(t,\xi,\zeta)-\xi\cdot \nabla_{\zeta}\widehat{u}(t,\xi,\zeta)=0,\ \widehat{u}(0,\cdot)=\widehat{u}^{0}.   \end{align*}
This is a linear transport equation which can be solved explicitly by $\widehat{u}(t,\xi,\zeta)=\widehat{u}^{0}(\xi,\zeta+t\xi)$. Consequently, we have 
\begin{align*}
\left\Vert \nabla_{\zeta}\widehat{u}(t,\cdot)\right\Vert_{\infty}=\left\Vert \nabla_{\zeta}\widehat{u}^{0}(\xi,\zeta+t\xi)\right\Vert_{\infty}= \left\Vert \nabla_{\zeta}\widehat{u}^{0}\right\Vert_{\infty}.        
\end{align*}
Note that since $u^{0}\in \mathcal{M}_{c}(\mathbb{R}^{d}\times \mathbb{R}^{d})$ the Fourier transform of $u^{0}$ has bounded derivatives and thus $\nabla_{\zeta}\widehat{u}^{0}\in L^{\infty}(\mathbb{R}^{d}\times \mathbb{R}^{d})$. 
\end{proof} 
\textit{Proof of Theorem \ref{Free transport limit}.} 
\textbf{Step 1}. Assume first  that $\widehat{u
}^{0}$ is real valued. The Fourier transform of $u_{\varepsilon}$ is a classical solution of the equation:
\begin{align}
\partial_{t}\widehat u_{\varepsilon}(t,\xi,\zeta)-\widehat{\eta}(\varepsilon(\xi,\zeta))\xi\cdot \nabla_{\zeta}\widehat u_{\varepsilon}(t,\xi,\zeta)=\varepsilon \xi \cdot\nabla_{\zeta}\widehat{\eta}(\varepsilon(\xi,\zeta))\widehat u_{\varepsilon}(t,\xi,\zeta),\ \widehat{u}_{\varepsilon}(0,\cdot)=\widehat{u}^{0} .  \label{equation for hatueps}
\end{align}
On the other hand, the Fourier transform of $u$ satisfies the equation:  
\begin{align}
\partial_{t}\widehat{u}(t,\xi,\zeta)-\xi\cdot \nabla_{\zeta}\widehat{u}(t,\xi,\zeta)=0, \ \widehat{u}(0,\cdot)=\widehat{u}^{0}.\label{equation for hat u}
\end{align}
Subtracting \eqref{equation for hat u} from  \eqref{equation for hatueps} we get 
\begin{align}
\begin{cases}
\begin{array}{lc}
\partial_{t}(\widehat{u}_{\varepsilon}-\widehat{u})+\widehat{\eta}(\varepsilon(\xi,\zeta))\xi \cdot \nabla_{\zeta}(\widehat{u}-\widehat{u}_{\varepsilon})+\xi\cdot \nabla_{\zeta}\widehat{u}-\widehat{\eta}(\varepsilon(\xi,\zeta))\xi\cdot \nabla_{\zeta}\widehat{u} =\varepsilon\xi \cdot \nabla_{\zeta}\widehat{\eta}(\varepsilon(\xi,\zeta))\widehat{u}_{\varepsilon}\\
(\widehat{u}_{\varepsilon}-\widehat{u})(0,\cdot)=0. 
\end{array}
\end{cases}
\label{equation for difference}
\end{align}
Multiplying \eqref{equation for difference} by $(\widehat{u}_{\varepsilon}-\widehat{u})w_{-s}(\xi,\zeta)$ and integrating we get 
\begin{align*}
\frac{\dd}{\dd t}\frac{1}{2}\left\Vert (u_{\varepsilon}-u)(t,\cdot)\right\Vert_{H^{-s}}^{2}&-\frac{1}{2}\int_{\mathbb{R}^{d}\times \mathbb{R}^{d}}w_{-s}(\xi,\zeta)\widehat{\eta}(\varepsilon(\xi,\zeta))\xi\cdot \nabla_{\zeta}\left\vert (\widehat{u}_{\varepsilon}-\widehat{u})\right\vert^{2}(t,\xi,\zeta)\ \dd \xi\dd \zeta\\
&+\int_{\mathbb{R}^{d}\times \mathbb{R}^{d}}w_{-s}(\xi,\zeta)\xi \cdot \nabla_{\zeta}\widehat{u}(t,\xi,\zeta)(1-\widehat{\eta}(\varepsilon(\xi,\zeta)))(\widehat{u}_{\varepsilon}-\widehat{u})(t,\xi,\zeta)\ \dd \xi\dd \zeta\\
&=\int_{\mathbb{R}^{d}\times \mathbb{R}^{d}}w_{-s}(\xi,\zeta)\varepsilon \xi\cdot \nabla_{\zeta}\widehat{\eta}(\varepsilon(\xi,\zeta))\widehat{u}_{\varepsilon}(t,\xi,\zeta)(\widehat{u}_{\varepsilon}-\widehat{u})(t,\xi,\zeta)\ \dd \xi \dd \zeta. 
\end{align*}
Therefore we get  
\begin{align*}
\frac{\dd}{\dd t}\left\Vert (u_{\varepsilon}-u)(t,\cdot)\right\Vert_{H^{-s}}^{2}=I_{1}+I_{2}+I_{3}     
\end{align*}
where we have set 
\begin{align*}
&I_{1}\coloneqq  \frac{1}{2}\int_{\mathbb{R}^{d}\times \mathbb{R}^{d}}w_{-s}(\xi,\zeta)\widehat{\eta}(\varepsilon(\xi,\zeta))\xi\cdot \nabla_{\zeta}\left\vert (\widehat{u}_{\varepsilon}-\widehat{u})\right\vert^{2}(t,\xi,\zeta)\ \dd \xi\dd \zeta\\
&I_{2}\coloneqq -\int_{\mathbb{R}^{d}\times \mathbb{R}^{d}}w_{-s}(\xi,\zeta)\xi \cdot \nabla_{\zeta}\widehat{u}(t,\xi,\zeta)(1-\widehat{\eta}(\varepsilon(\xi,\zeta)))(\widehat{u}_{\varepsilon}-\widehat{u})(t,\xi,\zeta)\ \dd \xi\dd \zeta
\end{align*}
and 
\begin{align*}
I_{3}\coloneqq \int_{\mathbb{R}^{d}\times \mathbb{R}^{d}}w_{-s}(\xi,\zeta)\varepsilon \xi\cdot \nabla_{\zeta}\widehat{\eta}(\varepsilon(\xi,\zeta))\widehat{u}_{\varepsilon}(t,\xi,\zeta)(\widehat{u}_{\varepsilon}-\widehat{u})(t,\xi,\zeta)\ \dd \xi \dd \zeta.     
\end{align*}
\textbf{Estimate on $I_{1}$.} Integrating by parts and observing that $-\nabla_{\zeta}w_{-s}(\xi,\zeta)=2sw_{-(s+1)}(\xi,\zeta)\zeta$ we obtain 
\begin{align*}
I_{1}=&-\frac{1}{2}\int_{\mathbb{R}^{d}\times \mathbb{R}^{d}}w_{-s}(\xi,\zeta)\varepsilon \xi \cdot \nabla_{\zeta}\widehat{\eta}(\varepsilon(\xi,\zeta))\left\vert \widehat{u}_{\varepsilon}-\widehat{u}\right\vert^{2}(t,\xi,\zeta)\ \dd \xi\dd  \zeta\\
&+s\int_{\mathbb{R}^{d}\times \mathbb{R}^{d}}w_{-(s+1)}(\xi,\zeta)\xi\cdot \zeta \widehat{\eta}(\varepsilon(\xi,\zeta))\left\vert \widehat{u}_{\varepsilon}-\widehat{u}\right\vert^{2}(t,\xi,\zeta)\ \dd \xi \dd \zeta\coloneqq I_{1}^{1}+I_{1}^{2}.      \end{align*}
To estimate $I_{1}^{1}$ note that
$\underset{(\xi,\zeta)\in \mathbb{R}^{d}\times \mathbb{R}^{d}}{\sup}\left\vert \varepsilon \xi\cdot \nabla_{\zeta}\widehat{\eta}(\varepsilon(\xi,\zeta))\right\vert= \underset{(\xi,\zeta)\in \mathbb{R}^{d}\times \mathbb{R}^{d}}{\sup} \left\vert\xi \cdot \nabla_{\zeta}\widehat{\eta}(\xi,\zeta)\right\vert$ and so we obtain the bound 
\begin{align}
I_{1}^{1}&\leq \frac{1
}{2}\underset{(\xi,\zeta)\in \mathbb{R}^{d}\times \mathbb{R}^{d}}{\sup}\left\vert \xi \cdot \nabla_{\zeta}\widehat{\eta}(\xi,\zeta)\right\vert \int_{\mathbb{R}^{d}\times \mathbb{R}^{d}}w_{-s}(\xi,\zeta)\left\vert \widehat{u}_{\varepsilon}-\widehat{u}\right\vert^{2}(t,\xi,\zeta)\ \dd \xi\dd \zeta \notag\\
&=  \frac{1
}{2}\underset{(\xi,\zeta)\in \mathbb{R}^{d}\times \mathbb{R}^{d}}{\sup}\left\vert \xi \cdot \nabla_{\zeta}\widehat{\eta}(\xi,\zeta)\right\vert \left\Vert (u_{\varepsilon}-u)(t,\cdot)\right\Vert_{H^{-s}}^{2}.     \label{est on I11}  
\end{align}
To estimate $I_{1}^{2}$, note that 
$$\left\vert \widehat{\eta}(\varepsilon(\xi,\zeta))w_{-(s+1)}(\xi,\zeta)\xi \cdot \zeta \right\vert\leq \left\Vert \widehat{\eta}\right\Vert_{\infty} w_{-s}(\xi,\zeta)\leq \left\Vert \eta\right\Vert_{1}w_{-s}(\xi,\zeta)=w_{-s}(\xi,\zeta).     
$$ 
Consequently, we obtain the bound  
\begin{align}
I_{1}^{2}\leq s\int_{\mathbb{R}^{d}\times \mathbb{R}^{d}}w_{-s}(\xi,\zeta)\left\vert \widehat{u}_{\varepsilon}-\widehat{u}\right\vert^{2}(t,\xi,\zeta)\ \dd \xi\dd \zeta=s\left\Vert (u_{\varepsilon}-u)(t,\cdot)\right\Vert_{H^{-s}}^{2}. \label{I21 est}  
\end{align}
Gathering \eqref{est on I11}-\eqref{I21 est} we conclude that 
\begin{align}
I_{1}\leq \left(s+\frac{1}{2}\underset{(\xi,\zeta)\in \mathbb{R}^{d}\times \mathbb{R}^{d}}{\sup}\left\vert \xi \cdot \nabla_{\zeta}\widehat{\eta}(\xi,\zeta)\right\vert \right)\left\Vert (u_{\varepsilon}-u)(t,\cdot)\right\Vert_{H^{-s}}^{2}. \label{I1 est}     
\end{align}
\textbf{Estimate on $I_{2}$.} By Lemma \ref{bound on negative weighted fourier} we have 
\begin{align*}
I_{2}&\leq \frac{1}{2}\int_{\mathbb{R}^{d}\times \mathbb{R}^{d}}w_{-s}(\xi,\zeta)\left\vert \widehat{u}_{\varepsilon}-\widehat{u}\right\vert^{2}(t,\xi,\zeta)\ \dd \xi\dd \zeta\\
&+\frac{1}{2}\underset{(t,\xi,\zeta)\in \mathbb{R}_{+}\times\mathbb{R}^{d}\times \mathbb{R}^{d}}{\sup}\left\vert \cdot \nabla_{\zeta}\widehat{u}(t,\xi,\zeta)\right\vert^{2} \int_{\mathbb{R}^{d}\times \mathbb{R}^{d}}w_{-s}(\xi,\zeta)\left\vert \xi\right\vert^{2}\left\vert 1-\widehat{\eta}(\varepsilon(\xi,\zeta))\right\vert^{2}\ \dd \xi\dd \zeta\\
&\leq \frac{1}{2}\left\Vert (u_{\varepsilon}-u)(t,\cdot)\right\Vert_{H^{-s}}^{2}+\frac{1}{2}\left\Vert \nabla_{\zeta}\widehat{u}^{0}\right\Vert_{\infty}^{2} \int_{\mathbb{R}^{d}\times \mathbb{R}^{d}}w_{-(s-1)}(\xi,\zeta)\left\vert (1-\widehat{\eta}(\varepsilon(\xi,\zeta)))\right\vert^{2}\ \dd \xi\dd \zeta .  
\end{align*}
Since $\widehat{\eta}(\varepsilon\cdot )\underset{\varepsilon\rightarrow 0}{\rightarrow} 1$ pointwise and $w_{-(s-1)}\in L^{1}(\mathbb{R}^{d}\times \mathbb{R}^{d})$ whenever $s>1+\frac{d}{2}$, we conclude from the dominated convergence theorem that 
\begin{align*}
\int_{\mathbb{R}^{d}\times \mathbb{R}^{d}}w_{-(s-1)}(\xi,\zeta)\left\vert 1-\widehat{\eta}(\varepsilon(\xi,\zeta))\right\vert^{2}\ \dd \xi\dd \zeta\underset{\varepsilon \rightarrow 0}{\rightarrow}0.    
\end{align*}
Thus, we deduce that 
\begin{align}
I_{2}\leq \frac{1}{2}
\left\Vert (u_{\varepsilon}-u)(t,\cdot)\right\Vert_{H^{-s}}^{2}+o_{\varepsilon}(1). \label{I2 est}   
\end{align}
\textbf{Estimate on $I_{3}$.} We recast $I_{3}$ as   
\begin{align*}
I_{3}&=\int_{\mathbb{R}^{d}\times \mathbb{R}^{d}}w_{-s}(\xi,\zeta)\varepsilon\xi \cdot \nabla_{\zeta}\widehat{\eta}(\varepsilon(\xi,\zeta))\left\vert \widehat{u}_{\varepsilon}-\widehat{u}\right\vert^{2}(t,\xi,\zeta)\ \dd \xi\dd \zeta\\
&+\int_{\mathbb{R}^{d}\times \mathbb{R}^{d}}w_{-s}(\xi,\zeta)\varepsilon \xi\cdot \nabla_{\zeta}\widehat{\eta}(\varepsilon(\xi,\zeta))\widehat{u}(t,\xi,\zeta)(\widehat{u}_{\varepsilon}-\widehat{u})(t,\xi,\zeta)\ \dd \xi\dd \zeta=I_{3}^{1}+I_{3}^{2}.    
\end{align*}
First, observe that 
\begin{align*}
I_{3}^{1}\leq\underset{(\xi,\zeta)\in \mathbb{R}^{d}\times \mathbb{R}^{d}}{\sup}\left\vert \xi \cdot \nabla_{\zeta}\widehat{\eta}(\xi,\zeta)\right\vert \left\Vert (u_{\varepsilon}-u)(t,\cdot)\right\Vert_{H^{-s}}^{2}.       
\end{align*}
Second, using the elementary inequality $\left\vert \xi \right\vert^{2}w_{-s}(\xi,\zeta)\leq w_{-(s-1)}(\xi,\zeta)$ and applying Lemma \ref{bound on negative weighted fourier} we get  
\begin{align*}
I^{2}_{3}&\leq \frac{\varepsilon^{2}}{2}\int_{\mathbb{R}^{d}\times \mathbb{R}^{d}}w_{-s}(\xi,\zeta)\left\vert \xi\right\vert^{2}\left\vert \nabla_{\zeta} \widehat{\eta}\right\vert^{2} (\varepsilon(\xi,\zeta))\left\vert \widehat{u}\right\vert^{2}(t,\xi,\zeta)\ \dd \xi\dd \zeta+\frac{1}{2}\left\Vert ({u}_{\varepsilon}-{u})(t,\cdot)\right\Vert_{H^{-s}}^{2}\\
&\leq \frac{\varepsilon^{2}}{2}\left\Vert \widehat{u}^{0}\right\Vert_{\infty}^{2}\left\Vert \nabla_{\zeta}\widehat{\eta}\right\Vert_{\infty}^{2}\int_{\mathbb{R}^{d}\times \mathbb{R}^{d}}w_{-(s-1)}(\xi,\zeta)\ \dd \xi\dd \zeta+\frac{1}{2}\left\Vert (u_{\varepsilon}-u)(t,\cdot)\right\Vert_{H^{-s}}^{2}.     
\end{align*}
Hence, we infer  
\begin{align}
I_{3}\lesssim \left\Vert (u_{\varepsilon}-u)(t,\cdot)\right\Vert_{H^{-s}}^{2}+o_{\varepsilon}(1).   \label{I3 est}   
\end{align}
To conclude, gathering \eqref{I1 est}-\eqref{I3 est} we infer that for some $C>0$ it holds that 
\begin{align*}
\frac{\dd}{\dd t}\left\Vert (u_{\varepsilon}-u)(t,\cdot)\right\Vert_{H^{-s}}^{2}\leq C\left(\left\Vert (u_{\varepsilon}-u)(t,\cdot)\right\Vert_{H^{-s}}^{2}+o_{\varepsilon}(1) \right).     
\end{align*}
By Gr\"onwall's Lemma we deduce that 
\begin{align}
\underset{t\in [0,T]}{\sup}\left\Vert (u_{\varepsilon}-u)(t,\cdot)\right\Vert_{H^{-s}}\lesssim o_{\varepsilon}(1)\underset{\varepsilon \rightarrow 0}{\rightarrow 0}. \label{Vanishing real initial data} \end{align}
Next, assume that 
$\widehat{u}^{0}$ is complex valued and denote $u_{R}=\Re u, u_{I}=\Im u$ and $u_{R}^{\varepsilon}=\Re u_{\varepsilon},u_{I}^{\varepsilon}=\Im u_{\varepsilon}$. Then $u_{R}^{\varepsilon}-u_{R}$ and $u_{I}^{\varepsilon}-u_{I}$ satisfy \eqref{equation for difference} and therefore thanks to \eqref{Vanishing real initial data} we deduce that  
\begin{align*}
\underset{t\in [0,T]}{\sup}\left\Vert (u_{\varepsilon}-u)(t,\cdot)\right\Vert_{H^{-s}}&\leq \underset{t\in [0,T]}{\sup}\left\Vert (u_{R}^{\varepsilon}-u_{R})(t,\cdot)\right\Vert_{H^{-s}}+\underset{t\in [0,T]}{\sup}\left\Vert (u_{I}^{\varepsilon}-u_{I})(t,\cdot)\right\Vert_{H^{-s}}\underset{\varepsilon \rightarrow 0}{\rightarrow} 0.  \end{align*}
\qed
\section{Proof of Theorem \ref{nonlocal to local general B}}\label{general B sec}
We start by proving preservation in time of positivity. In the settings of nonlocal conservation laws (such as \eqref{nonlocal conservation}) this is a basic feature of the equation which does not necessitate any anisotropic  structure on the kernel. However, in the present settings we take advantage of the anisotropic structure to prove this property.  
\begin{lem}\label{Positivity invariant}
Let the assumptions and notations of Theorem \ref{nonlocal to local general B} hold. Then $u_{\varepsilon}\in \mathcal{M}_{+}([0,T]\times \mathbb{R})$.     
\end{lem}
\begin{proof} \textbf{Step 1}. Suppose first that $u^{0}\in W^{1,\infty}(\mathbb{R})$. Moreover, by a routine approximation argument we may assume that $\eta\in W^{1,\infty}(\mathbb{R}_{-})$ (see for instance Theorem 1.1, step 3 in \cite{colombo2023nonlocal}). Let $u_{\varepsilon}$ be the unique classical solution with initial data $u^{0}$ to \eqref{nonlocal Cauhy problem simplification} guaranteed via Theorem \ref{wellposedness classical sol}. Multiplying \eqref{nonlocal Cauhy problem simplification} by $\mathrm{sgn}(u_{\varepsilon}(t,x))$ we get 
\begin{align}
\partial_{t}\left\vert u_{\varepsilon}\right\vert+\mathrm{sgn}(u_{\varepsilon})B\partial_{x}(\eta_{\varepsilon}\ast u_{\varepsilon})=0, \ \left\vert u_{\varepsilon}\right\vert (0,\cdot)=\left\vert u^{0}\right\vert. \label{equation for absolute value}  \end{align}
Subtracting \eqref{nonlocal Cauhy problem simplification} from \eqref{equation for absolute value} we get 
\begin{align}
\partial_{t}(\left\vert u_{\varepsilon}\right\vert(t,x)-u_{\varepsilon}(t,x))&=B\partial_{x}(\eta_{\varepsilon}\ast u_{\varepsilon})-\mathrm{sgn}(u_{\varepsilon}(t,x))B\partial_{x}(\eta_{\varepsilon}\ast u_{\varepsilon}) \notag\\
&=B(1-\mathrm{sgn}(u_{\varepsilon}(t,x)))\left[\frac{1}{\varepsilon^{2}}\int_{x}^{\infty}\eta'(\frac{x-y}{\varepsilon})u_{\varepsilon}(t,y)\ \dd y-\frac{u_{\varepsilon}(t,x)}{\varepsilon}\right] \notag\\
&\leq B\left[\frac{\left\vert u_{\varepsilon}\right\vert(t,x)-u_{\varepsilon}(t,x)}{\varepsilon}-\frac{1}{\varepsilon^{2}}\int_{x}^{\infty}\eta'(\frac{x-y}{\varepsilon})(\left\vert u_{\varepsilon}\right\vert-u_{\varepsilon})(t,y)\ \dd y \right], \label{ine for difference}
\end{align}
where we have used that $B\leq 0$ in the last inequality. 
If $x$ is a point of maximum of $\left\vert u_
{\varepsilon}\right\vert(t,x)-u_{\varepsilon}(t,x)$ then we have 
\begin{align*}
\frac{\left\vert u_{\varepsilon}\right\vert(t,x)-u_{\varepsilon
}(t,x)}{\varepsilon}-\frac{1}{\varepsilon^{2}}\int_{x}^{\infty}\eta'(\frac{x-y}{\varepsilon})(\left\vert u_{\varepsilon}\right\vert-u_{\varepsilon})(t,y)\ \dd y\geq 0.     
\end{align*}
Therefore, evaluating inequality \eqref{ine for difference} at a point of maximum of $\left\vert u_{\varepsilon}\right\vert (t,x)-u_{\varepsilon}(t,x)$ and using that $B\leq 0$ and that $u^{0}\geq 0$ we get 
\begin{align*}
\frac{\dd}{\dd t}\left\Vert (\left\vert u_{\varepsilon}\right\vert -u_{\varepsilon})(t,\cdot)\right\Vert_{\infty}\leq 0\Rightarrow \left\Vert (\left\vert u_{\varepsilon}\right\vert -u_{\varepsilon})(t,\cdot)\right\Vert_{\infty}=0\Rightarrow u_{\varepsilon}\equiv \left\vert u_{\varepsilon}\right\vert.        
\end{align*}  
\textbf{Step 2}. Assume now that  $u^{0}\in H^{-s}(\mathbb{R})\cap \mathcal{M}_{+}(\mathbb{R})$. Fix a family of standard mollifiers $\{\chi_{\delta}\}_{\delta>0}$ and consider the regularized problem 
\begin{align*}
\partial_{t}u_{\varepsilon}^{\delta}+B\partial_{x}(\eta_{\varepsilon}\ast u_{\varepsilon}^{\delta})=0, \ u_{\varepsilon}^{\delta}(0,\cdot)=\chi_{\delta}\ast u^{0}.     
\end{align*}
Then $\chi_{\delta}\ast u^{0}\in W^{1,\infty}(\mathbb{R})$ and $\chi_{\delta}\ast u^{0}\geq 0$. Therefore, by what we already proved in step 1 it holds that $u_{\varepsilon}^{\delta}(t,x)\geq 0$ for all $(t,x)\in \mathbb{R}_{+}\times \mathbb{R}$. By Remark \ref{Remark about approximating measure valued solutions} we have $u_{\varepsilon}^{\delta}\underset{\delta \rightarrow 0}{\rightharpoonup} u_{\varepsilon}$ in $L^{2}([0,T];H^{-s}(\mathbb{R}))$ where $u_{\varepsilon}$ is the unique distributional solution of \eqref{nonlocal Cauhy problem simplification} with initial data $u^{0}$. Therefore for any $\varphi\in C^{\infty}_{0}((-a,T)\times \mathbb{R})$ we have 
\begin{align*}
0\leq \int_{0}^{T}\langle \varphi,u_{\varepsilon}^{\delta}\rangle\ \dd t \underset{\delta \rightarrow 0}{\rightarrow}\int_{0}^{T}\langle \varphi,u_{\varepsilon}\rangle\ \dd t.    
\end{align*}
By the Schwartz representation theorem any positive distribution is a positive measure and therefore we deduce that   $u_{\varepsilon}\in \mathcal{M}_{+}([0,T]\times \mathbb{R})$. 
\end{proof}
\textit{Proof of Theorem \ref{nonlocal to local general B}}.  Consider the regularized problem 
\begin{align*}
\partial_{t}u_{\varepsilon}^{\delta}+B\partial_{x}(\eta_{\varepsilon}\ast u_{\varepsilon}^{\delta})=0, \ u_{\varepsilon}^{\delta}(0,\cdot)=\chi_{\delta}\ast u^{0}.     
\end{align*}
By Lemma \ref{Positivity invariant} we have 
$u_{\varepsilon}^{\delta}\geq 0$ and therefore 
\begin{align*}
\frac{\dd}{\dd t}\left\Vert u_{\varepsilon}^{\delta}(t,\cdot)\right\Vert_{1}&=\frac{\dd}{\dd t}\int_{\mathbb{R}}u_{\varepsilon}^{\delta}(t,x)\ \dd x=-\int_{\mathbb{R}}B(x)\partial_{x}(\eta_{\varepsilon}\ast u_{\varepsilon}^{\delta})(t,x)\ \dd x\\
&=\int_{\mathbb{R}}\partial_{x}B(x)\eta_{\varepsilon}\ast u_{\varepsilon}^{\delta}(t,x)\ \dd x\leq \left\Vert \partial_{x}B\right\Vert_{\infty}\left\Vert \eta\right\Vert_{1}\left\Vert u_{\varepsilon}^{\delta}(t,\cdot)\right\Vert_{1}=\left\Vert \partial_{x}B\right\Vert_{\infty}\left\Vert u_{\varepsilon}^{\delta}(t,\cdot)\right\Vert_{1}.    
\end{align*}
By Gr\"onwall's inequality we conclude that 
\begin{align}
\left\Vert u_{\varepsilon}^{\delta}(t,\cdot)\right\Vert_{1}\leq e^{\left\Vert \partial_{x}B\right\Vert_{\infty}t}\left\Vert \chi_{\delta}\ast u^{0}\right\Vert_{1}\leq e^{\left\Vert \partial_{x}B\right\Vert_{\infty}t}\left\vert u^{0}\right\vert. \label{bound on L1 norm of delta regularization}      
\end{align}
Since $L^{1}(\mathbb{R})$ is embedded in $H^{-s}(\mathbb{R})$ for any $s>\frac{1}{2}$ it follows that there is some $u_{\varepsilon}\in L^{\infty}([0,T];H^{-s}(\mathbb{R}))$ such that $u_{\varepsilon}^{\delta}\underset{\delta \rightarrow 0}{\rightharpoonup}u_{\varepsilon}$ weakly in $L^{2}([0,T];H^{-s}(\mathbb{R}))$. 
By uniqueness, it must be that $u_{\varepsilon}$ is the unique distributional solution of \eqref{nonlocal Cauhy problem simplification}. Moreover, by \eqref{bound on L1 norm of delta regularization} it holds that 
\begin{align*}
\left\Vert u_{\varepsilon}\right\Vert_{L^{\infty}([0,T];H^{-s}(\mathbb{R}))}\leq e^{T\left\Vert \partial_{x}B\right\Vert_{\infty}}\left\vert u^{0}\right\vert.      
\end{align*}
Therefore, again by Lemma \ref{BanachAl Lemma} there exist some $u\in L^{\infty}([0,T];H^{-s}(\mathbb{R}))$ and a sequence $\varepsilon_{k}\underset{k\rightarrow \infty}{\rightarrow} 0$ such that it holds that $u_{\varepsilon_{k}}\underset{k\rightarrow \infty}{\rightharpoonup} u$ weakly in $L^{2}([0,T];H^{-s}(\mathbb{R}))$. Explicitly, this means that for any $\phi\in L^{2}([0,T];H^{s}(\mathbb{R}))$ it holds that 
\begin{align*}
\int_{0}^{T} \langle \phi,u_{\varepsilon_{k}}\rangle\ \dd \tau  \underset{k\rightarrow \infty}{\rightarrow} \int_{0}^{T}\langle \phi,u\rangle\ \dd \tau.  
\end{align*}
Based on this, we proceed by proving that for any $\varphi \in C^{\infty}_{0}(
(-a,T)\times \mathbb{R})$ it holds that 
\begin{align*}
\int_{0}^{T}\langle  \partial_{x}(\eta^{\mathrm{odd}}_{\varepsilon_{k}}\ast(B\varphi)),u_{\varepsilon_{k}}\rangle \ \dd \tau\underset{k \rightarrow \infty}{\rightarrow} 0    
\end{align*}
and 
\begin{align*}
\int_{0}^{T}\langle  \partial_{x}(\eta_{\varepsilon_{k}}^{\mathrm{even}}\ast (B\varphi)), u_{\varepsilon_{k}}\rangle \underset{k\rightarrow \infty}{\rightarrow}\int_{0}^{T}\langle \partial_{x}(B\varphi),u\rangle \ \dd \tau.     
\end{align*}
For any  test function $\varphi\in C^{\infty}_{0}((-a,T)\times \mathbb{R})$ and any given approximation of identity $\rho_{\varepsilon}$ we have 
\begin{align*}
&\left\vert \int_{0}^{T}\langle \rho_{\varepsilon_{k}}\ast \partial_{x}(B\varphi), u_{\varepsilon_{k}}\rangle\ \dd \tau -\int_{0}^{T}  \langle \partial_{x}(B\varphi), u\rangle \ \dd \tau\right\vert\\
&\leq \left\vert \int_{0}^{T} \langle\rho_{\varepsilon_{k}}\ast \partial_{x}(B\varphi)-\partial_{x}(B\varphi), u_{\varepsilon_{k}}\rangle\ \dd \tau  \right\vert\\
&+\left\vert  \int_{0}^{T} \langle \partial_{x}(B\varphi),u_{\varepsilon_{k}}-u\rangle \ \dd \tau\right\vert\coloneqq I_{k}+J_{k}. 
\end{align*}
Clearly $\partial_{x}(B\varphi)\in L^{2}([0,T];H^{s}(\mathbb{R}))$ and therefore $J_{k}\underset{k \rightarrow \infty}{\rightarrow}0$. In addition, since $\left\Vert u_{\varepsilon}\right\Vert_{L^{2}([0,T];H^{-s}(\mathbb{R}))}$ is uniformly bounded in $\varepsilon$ we have 
\begin{align*}
I_{k}\leq \left\Vert \rho_{\varepsilon_{k}}\ast \partial_{x}(\varphi B)-\partial_{x}(B\varphi)\right\Vert_{L^{2}([0,T];H^{s}(\mathbb{R}))}\left\Vert u_{\varepsilon_{k}}\right\Vert_{L^{2}([0,T];H^{-s}(\mathbb{R}))}\underset{k\rightarrow \infty}{\rightarrow}0.       
\end{align*}
 It follows that
\begin{align}
\int_{0}^{T}\langle \rho_{\varepsilon_{k}}\ast \partial_{x}(B\varphi),u_{\varepsilon_{k}}\rangle\ \dd \tau \underset{k\rightarrow \infty}{\rightarrow}\int_{0}^{T}  \langle \partial_{x}(B\varphi),u \rangle\ \dd \tau. \label{Convregence for rhoeps} \end{align}
Note that if $\eta_{\varepsilon}$ is an approximation of the identity then so is 
$\eta_{\varepsilon}(-x)$. Therefore, applying \eqref{Convregence for rhoeps} with $\rho_{\varepsilon}(x)=\eta_{\varepsilon}(x)$ and $\rho_{\varepsilon}(x)=\eta_{\varepsilon}(-x)$ we deduce that 
\begin{align*}
\int_{0}^{T} \langle \eta_{\varepsilon_{k}}^{\mathrm{odd}}\ast \partial_{x}(B\varphi),u_{\varepsilon_{k}}\rangle \ \dd \tau\underset{k \rightarrow \infty}{\rightarrow} 0    
\end{align*}
and 
\begin{align*}
\int_{0}^{T} \langle \eta_{\varepsilon_{k}}^{\mathrm{even}}\ast \partial_{x}(B\varphi), u_{\varepsilon_{k}}\rangle \ \dd \tau\underset{k\rightarrow \infty}{\rightarrow}\int_{0}^{T} \langle \partial_{x}(B\varphi), u\rangle\ \dd \tau .     
\end{align*}
It follows that $u\in L^{\infty}([0,T];H^{-s}(\mathbb{R}))$ is a distributional solution to \eqref{local Cauchy}, and so by uniqueness it must be that $u$ is the unique distributional solution of \eqref{local Cauchy}. The uniqueness of $u$ also implies that all converging subsequences of $u_{\varepsilon}$ converge to the same limit, and so the convergence must hold on the entire sequence. \qed 
\vspace{0.3cm}

\noindent{\bf Acknowledgments.} The author would like to thank Gianluca Crippa and Laura Spinolo for their interest in this work and for several stimulating discussions. The author indebted to the University of Basel for financial support. 

\bibliographystyle{abbrv}
\bibliography{references}
\end{document}